\documentclass[12pt,reqno]{amsart}
\DeclareFontFamily{OML}{script}{}
\DeclareFontShape{OML}{script}{m}{it}
{ <5-20> rsfs10 }{}
\DeclareMathAlphabet{\mathscript}{OML}{script}{m}{it}

\renewcommand{\mathcal}[1]{{\mathscript #1}\hspace{0.2ex}}
\usepackage{cite}
\usepackage{color}
\ifx\red\undefined
\newcommand{\red}{\color{red}}

\fi
\usepackage[ansinew]{inputenc}
\usepackage[encapsulated]{CJK}

\usepackage{graphicx,graphics}
\usepackage{cite}
\usepackage{pifont}
\usepackage{amsthm}
\usepackage{subfigure}
\usepackage{amscd}
\usepackage{amsmath}
\usepackage{latexsym}
\usepackage{amsfonts}
\usepackage{amssymb}
\usepackage{color}
\usepackage{multicol}
\usepackage{amsmath,amssymb,amsthm,amsfonts,mathrsfs}
\usepackage{hyperref}
\hypersetup{hypertex=true,
            colorlinks=true,
            linkcolor=blue,
            anchorcolor=blue,
            citecolor=blue}
\usepackage{bm}

\usepackage{tikz}
\usepackage{pgfplots}
\usetikzlibrary{patterns}
\usepackage{enumitem}
\makeatletter
\renewcommand{\l@subsection}{\@tocline{2}{0pt}{2.8em}{}{}}%
\renewcommand{\l@subsubsection}{\@tocline{3}{0pt}{5.4em}{}{}}%
\makeatother

\definecolor{ocre}{RGB}{64,123,121}

\allowdisplaybreaks[4]
\ifx\text\undefined
\newcommand{\text}{\mbox}
\fi
\ifx\operatorname\undefined
\newcommand{\operatorname}{\mathop}
\fi

\allowdisplaybreaks[3]

\newtheorem{theorem}{Theorem}[section]
\newtheorem{lemma}[theorem]{Lemma}
\newtheorem{proposition}[theorem]{Proposition}

\newtheorem{corollary}[theorem]{Corollary}

\theoremstyle{remark}

\def\R{\mathbb{R}}

\numberwithin{equation}{section}

\begin{document}
\title{Gradient estimates and Liouville-type theorems for the semilinear elliptic equation involving the nonlinear gradient source}
\author{Wenguo Liang and Zhengce Zhang}
\date{\today}
\address[Wenguo Liang]{School of Mathematics and Statistics, Xi'an Jiaotong University,
Xi'an, 710049, P. R. China}
\email{liangwenguo@stu.xjtu.edu.cn}
\address[Zhengce Zhang]{School of Mathematics and Statistics, Xi'an Jiaotong University,
Xi'an, 710049, P. R. China}
\email{zhangzc@mail.xjtu.edu.cn}
\thanks{Corresponding author: Zhengce Zhang}
\thanks{ }

\begin{abstract}
We study local and global properties of positive solutions to the equation $-\Delta u=u^p+M|\nabla u|^q$ in a domain $\Omega$ of $\mathbb R^N$, where $p,q$ are parameters and $M>0$. By constructing a linear operator, we establish the differential inequality containing an auxiliary function. By selecting appropriate auxiliary functions over various regions and employing the maximum principle, we derive the local gradient estimates for all $(p,q)\in \mathbb R^2$, and further establish Liouville-type theorems. As an application, we acquire universal estimates for local solutions of elliptic equations with general nonlinearities.  Our results extend partial conclusions established in Bidaut-V\'{e}ron, Garcia-Huidobro and V\'{e}ron [Math. Ann. 378 (1-2) (2020) 13-56]. \\
\end{abstract}

\maketitle

{\tableofcontents}  

\section{Introduction}
The article is concerned with local and global properties of positive solutions to the elliptic equation
\begin{align}\label{eq1}
-\Delta u=u^p+M|\nabla u|^q
\end{align}
in $\Omega$, where $\Omega\subset\R^N$ is a domain with $N\geq 2$, $M>0$, and $p,q$ are parameters. For $A\geq 0$, let
\begin{equation*}
\Omega_A:=\{x\in\Omega;\,|\nabla u(x)|=A\}
\end{equation*}
be the level set of gradient. The term $|\nabla u|^q$ is singular with $q<0$ on $\Omega_0$. A function $u\in C^1(\Omega)$ is said to be a weak solution of \eqref{eq1} if
\begin{equation*}
  |\nabla u|^q\in L^1_{\rm loc}(\Omega),
\end{equation*}
and
\begin{equation*}
  \int_{\Omega}\langle \nabla u,\nabla\psi\rangle=\int_{\Omega}u^p\psi+M\int_{\Omega}|\nabla u|^q\psi\quad\text{for all}\ \psi\in C^{\infty}_0(\Omega).
\end{equation*}

When $M=0$, equation \eqref{eq1} reduces to the classical Lane--Emden equation
\begin{equation}\label{Lane-Emden}
  -\Delta u=u^p.
\end{equation}
Gidas and Spruck \cite{Gidas-Spruck} used integral estimates to establish the nonexistence of nontrivial solutions in $\mathbb R^N$ for $1<p<p_S:=(N+2)/(N-2)$ with $N>2$. Bidaut--V\'{e}ron \cite{Veron-ARMA-89} proved an analogous nonexistence result in an exterior domain if $1<p<N/(N-2)$. Later, in an arbitrary domain $\Omega\subset \mathbb R^N$, Dancer \cite{Dancer-Math.Z-98} presented the priori estimates of solutions in terms of the distance to the boundary. Serrin and Zou \cite{Serrin-Zou-Acta} observed that the nonexistence of nontrivial solutions in $\mathbb R^N$ can be seen as a limiting case of universal boundedness theorems. Pol\'{a}\v{c}ik, Quittner and Souplet \cite{Polacik-Quittner-Souplet} used scaling arguments to establish new connections between universal estimates and Liouville-type theorems. While Wang and Wei \cite{Wang-Wei-JDE-23} applied the Nash--Moser iteration method to prove the nonexistence positive solutions to  \eqref{Lane-Emden} over a complete Riemannian manifold, where the exponent $p$ is permitted to be negative.

For problems whose nonlinearity contains only a gradient term, the equation reduces to the well-known Hamilton--Jacobi equation or Riccati equation
\begin{equation}\label{Hami-Joc}
-\Delta u=M|\nabla u|^q.
\end{equation}
When $M>0$, Lions \cite{Lions-JAM-85} used the Bernstein technique to derive the local gradient estimates, and proved that any classical solution in $\mathbb R^N$ with $q>1$ must be a constant. When $M<0$,  the pointwise gradient estimates was established by Bidaut-V\'{e}ron, Garcia-Huidobro and V\'{e}ron \cite{Verom-Huidobro-VeronL-JFA} based upon a combination of the Bernstein method and the Keller--Osserman's estimates for $q>1$. While for the half-space case, Porrentta and V\'{e}ron \cite{Porretta-Verom-Adv-Stu-06} obtained the Liouville-type classification for $1<q\leq 2$. Filippucci, Pucci and Souplet proved an analogous symmetry result for $q>2$ and $M>0$ in \cite{Filippucci-Pucci-Souplet-CPDE}, based on the moving planes method together with Bernstein estimates and compactness arguments. For extra qualitative properties and priori estimates of equations  related to \eqref{Hami-Joc}, we refer readers to \cite{Aghajani-CVPDE-21, blow-up1,Chang-Hu-Zhang-JDE-2023,Chang-Ju-Zhang-DCDS-2020,Ferreira-JAM-16, Lions-Math-Ann,Tommaso-Porretta-ARMA}.

Equation \eqref{eq1} with $M<0$ proposed in \cite{Chipot-Weissler-SIAM} is related to the  Chipot--Weissler equation. This equation models  population dynamics, describing how the population evolves under a given natural regulatory mechanism. The  key interest of \eqref{eq1} lies in the presence of two nonlinearities acting in opposite directions.
By using the Bernstein method, Bidaut-V\'{e}ron, Garcia-Huidobro and V\'{e}ron \cite{Veron-MathAnn-20} obtained the priori estimates and nonexistence of solutions in the range $\min\{p,q\}>1$ and $M\in\mathbb R$. Recently, He and Wang \cite{He-Wang-arXive-2024} adopted the Nash--Moser iteration to derive local estimates and Liouville-type theorems for $\max\{p,q\}<(N+3)/(N-1)$. For the case $M>0$, by using the integral identity and Young's inequality, Ma, Wu and Zhang \cite{Ma-Wu-arXive-24} have recently completed the nonexistence results of solutions in the critical case, that is $q=2p/(p+1)$. Further results on gradient estimates and the nonexistence of solutions to equations associated with \eqref{eq1} are available in \cite{Veron-CVPDE-23,Chang-Zhang-NA-2026,Han-Wang-JFA,He-Wang-arXive-2024,Liang-Zhang-CVPDE}.

The paper is devoted to the gradient estimates and Liouville-type theorems for equation \eqref{eq1} with $M>0$. Before proving the main results, we first introduce several preliminary notes.  Let $u$ be a solution of \eqref{eq1}. We define
\begin{equation*}
  u_\lambda(x)=\lambda^{\frac2{p-1}}u(\lambda x),\quad\lambda>0.
\end{equation*}
Clearly, $u_\lambda$ is a solution of the equation
\begin{equation*}
  -\Delta u=u^p+M\lambda^{\frac{2p-q(p+1)}{p-1}}|\nabla u|^q.
\end{equation*}
Then equation \eqref{eq1} is scale invariant for any $\lambda>0$ if and only if $q=2p/(p+1)$. In this context, we refer to $q$ the critical exponent with respect to $p$. Correspondingly,  $q$ is called subcritical if $q<2p/(p+1)$, and supercritical if $q>2p/(p+1)$.

The primary objective of the present paper are twofold. Firstly, we aim to establish Liouville-type theorems of \eqref{eq1} for all $(p,q)\in \mathbb R^2$. Secondly, as applications, we intend to obtain the universal estimates of positive solutions to \eqref{eq1}.

Compared with previous results obtained in \cite{Veron-MathAnn-20} with $\min\{p,q\}>1$, and in \cite{He-Wang-arXive-2024} with $\max\{p,q\}<(N+3)/(N-1)$, the gradient estimates of solutions to \eqref{eq1} for $(p,q)\in\mathbb R^2$ is more complicated and thus more delicate analytical technique are required. Actually, on the one hand, the key Keller--Osserman type estimates  \cite[Theorem A]{Veron-MathAnn-20} requires $q>1$. On the other hand, the auxiliary function proposed in \cite{He-Wang-arXive-2024,Q-W-J-CVPDE-2026} does not work to achieve the desired estimates. In our setting, by constructing a linear differential operator, together with introducing parameters, we establish differential inequalities on the superlevel sets of gradients via choosing auxiliary functions, and further obtain Liouville-type theorems. Furthermore, we derive the universal estimates through rescaling and doubling arguments to elliptic equations with general nonlinearities.

To better illustrate the main contributions and ideas of this paper, we shall elaborate on these methods in detail as follows.
\begin{enumerate}[
    label=(\roman*),
    leftmargin=*,
    labelsep=0.5em,        
]
\item \textit{Construction of a linear differential operator.} The Bernstein method depends on a differential inequality associated with $w$, where $w=|\nabla v|^2$, $v=f^{-1}(-u)$ and $f$ is a suitable auxiliary function. For linear operator $\mathcal L:=-\Delta +\mathcal H\cdot\nabla $ with $\mathcal H\in\mathbb R^N$, we observe that
\begin{equation*}
  \mathcal L(u^\alpha w^\gamma)=w^\gamma\mathcal L(u^\alpha)+u^\alpha \mathcal L(w^\gamma)-2\nabla u^\alpha\cdot\nabla w^\gamma.
\end{equation*}
It follows from \eqref{eq1} that
\begin{equation*}
  -\Delta u^\alpha=-\alpha(\alpha-1)u^{\alpha-2}|\nabla u|^2-\alpha u^{\alpha-1}\Delta u<0
\end{equation*}
with $\alpha<0$, and
\begin{equation*}
  -\Delta w^\gamma=-\gamma(\gamma-1)w^{\gamma-2}|\nabla w|^2-\gamma w^{\gamma-1}\Delta  w.
\end{equation*}
Thus, the term $-w^\gamma \Delta u^\alpha$ arising from $w^\gamma\mathcal L(u^\alpha)$ and the non-positive term $-\gamma(\gamma-1)|\nabla w|^2$ with $\gamma>1$ contribute to deriving $\mathcal L(u^\alpha w^\gamma)<0$ for $w>0$. Nevertheless, we still need to handle additional terms $w^\gamma\mathcal H\cdot\nabla u^\alpha$, $u^\alpha\mathcal H\cdot\nabla w^\gamma$ and $-2\nabla u^\alpha\cdot\nabla w^\gamma$. To this end, we rewrite $w^\gamma\mathcal H\cdot\nabla u^\alpha$ and $u^\alpha \mathcal H\cdot\nabla w^\gamma$ as $\mathcal H_\alpha\cdot \nabla (u^\alpha w^\gamma)$. Therefore, the operator $\mathcal L$ can be replaced by
\begin{equation}\label{L-alpha-Intru}
  \mathcal L_\alpha(z):=-\Delta z+\mathcal H_\alpha \cdot\nabla z
\end{equation}
with
\begin{equation}\label{H-alpha-Intru}
 \mathcal H_\alpha:=\mathcal H+2\alpha \frac{f'}f\nabla v= \left(qM\frac{|f'|^q}{f'}w^{\frac{q-2}2}-2\frac{f''}{f'}+2\alpha \frac{f'}f\right)\nabla v.
\end{equation}

\item \textit{The subsolution of $\mathcal L_\alpha(z)=0$ on the superlevel set of gradient.} The local pointwise gradient estimates relies on the inequality
    \begin{equation*}
      \mathcal L_\alpha(u^\alpha w^\gamma\eta)\leq -(u^\alpha w^\gamma\eta)^\theta+C(\eta),
    \end{equation*}
where $\eta$ is a cut-off function and $\theta,C(\eta)>0$. However, for given $x_0\in\Omega$, it seems that this inequality does not hold in $\{x\in B_{3R/4};\, (u^\alpha w^\gamma\eta)(x)>0\}$ for $\alpha<0$, where $B_R=B_R(x_0)$ is the open ball with center $x_0$ and radius $R={\rm dist}(x_0,\partial\Omega)$. Fortunately, according to the interaction among $u^p$, $M|\nabla u|^q$ and $\Delta u$,  it holds that \begin{equation*}
  \mathcal L_\alpha (u^\alpha w^\gamma\eta)<0 \quad \text{in}\ \{x\in B_{3R/4};\, (u^\alpha w^\gamma\eta)(x)>C(\eta)\}
\end{equation*}
when $u$ admits a lower or upper bound with respect to $M$.  By virtue of the maximum principle on the superlevel set of $u^\alpha w^\gamma\eta$, we obtain the estimates for $u^\alpha w^\gamma$ in $\Omega$. The idea of working on the superlevel set of $u^\alpha w^\gamma\eta$ is inspired by the argument proposed by Cirant and Goffi \cite{Girant-Goffi-ARMA-2021}; see also \cite{Cianchi-CPDE-2011,Grenon-CRMA-2006} and the references therein.
\end{enumerate}

In addition, we establish local gradient estimates and Liouville-type result of bounded solutions for $q=2p/(p+1)$ with $p>1$. Noting that equation \eqref{eq1} is scale invariant, we therefore derive universal estimates for positive solutions by means of rescaling and doubling methods. These estimates guarantee the validity of Liouville-type results for positive solutions without the boundedness assumption.

Now we summarise our main results. Let $(p,q)\in\mathbb \R^2=\cup_{i=1}^6 G_i$, where sets $G_i(i=1,2,\ldots,6)$ in Figure \ref{fig:parameter_space} are defined as follows:
\begin{align*}
  G_1=&\bigg\{(p,q);\,p\geq \frac{N+3}{N-1},\ q<\frac{2p}{p+1}\bigg\};\quad &&G_2=\bigg\{(p,q);\, p\geq\frac{N+3}{N-1},\ q>\frac{2p}{p+1}\bigg\};\\
G_3=&\bigg\{(p,q);\,p<\frac{N+3}{N-1},\ q<\frac{N+1}{N-1}\bigg\};\quad &&G_4=\bigg\{(p,q);\,p\leq 0,\ q\geq \frac{N+1}{N-1}\bigg\};\\
G_5=&\bigg\{(p,q);\,0<p<\frac{N+3}{N-1},\ q\geq \frac{N+1}{N-1}\bigg\};\quad &&G_6=\bigg\{(p,q);\, p>1,\ q=\frac{2p}{p+1}\bigg\}.
\end{align*}
The first result concerns the local pointwise gradient estimates for solutions to \eqref{eq1} in the case that $q$ is subcritical with respect to $p\geq (N+3)/(N-1)$.

\begin{theorem}\label{theorem:p>N+3,q<1}
Let $(p,q)\in G_1$. Assume that $u$ is a positive solution of \eqref{eq1} in $\Omega$ satisfying
\begin{equation}\label{u<M1*}
  u\leq c_{N,p,q}M^{\frac{2}{2p-(p+1)q}}
\end{equation}
for some $c_{N,p,q}>0$. Then for $\max\{-(3N)^{-1},-4(3-q)^{-1}\} <\alpha<0$, there exists a constant $C=C(N,M,p,q,\alpha)>0$ such that
\begin{equation*}
  \bigl|\nabla u^{\frac{\alpha+2}2}(x)\bigl|\leq C \left({\rm dist}^{-1}(x,\partial\Omega)+{\rm dist}^{-\frac1{3-q}}(x,\partial\Omega)\right),\quad x\in\Omega.
\end{equation*}
Consequently, \eqref{eq1} possesses no positive solution in $\mathbb R^N$ satisfying \eqref{u<M1*}.
\end{theorem}

We first introduce an auxiliary function $f$ with $f'\neq 0$, and study the transformation $u=-f(v)$. Setting $w=|\nabla v|^2$, we derive a differential inequality for $u^\alpha w$, instead of one for $w$, as follows
\begin{align}\label{u-alpha-Lu-alpha}\nonumber
  u^{-\alpha}\mathcal L_\alpha(u^\alpha w)\leq &\left(C+\frac{f''}{(f')^2}u \right)u^{p-1}w +C\left(|f'|^{q-2}f''+\alpha u^{-1}|f'|^q\right)w^{\frac{q+2}2}\\
  &+\left(\alpha(\alpha+1)u^{-2}(f')^2+2\alpha f'' u^{-1}\right)w^2+2\left(\frac{f''}{f'}\right)'w^2-2|D^2 v|^2
\end{align}
with $C=C(N,M,p,q,\alpha)>0$, where $\mathcal L_\alpha$ and $\mathcal H_\alpha$ are given by \eqref{L-alpha-Intru} and \eqref{H-alpha-Intru}, respectively.
Since $q$ is subcritical and hence $q<2$, the terms containing $w^2$ on the right-hand side of \eqref{u-alpha-Lu-alpha} serve as dominate negative contributions for a suitable function $f$. With a delicate choice of $f$, the first two terms on the right-hand side of \eqref{u-alpha-Lu-alpha} can be partially absorbed by $w^2$ and $|D^2v|^2$ arising from $(\Delta u)^2$. Subsequently, thanks to introducing the parameter $-1<\alpha<0$ and utilizing the boundedness of $u$, we arrive at
\begin{equation*}
  2u^{-\alpha}\mathcal L_\alpha(u^\alpha w)\leq \alpha(\alpha+1)u^{-2}(f')^2w^2< 0
\end{equation*}
for $w>0$. Once this is done, replacing $\mathcal L_\alpha(u^\alpha w)$ by $\mathcal L_\alpha(u^\alpha w\eta)$ preserves the validity of previous arguments for a given cut-off function $\eta$. Setting $z=u^\alpha w\eta$, it follows that
\begin{equation*}
  \mathcal L_\alpha(z)<0\quad \text{in}\ \Omega':=\{x\in B_{3R/4};\,z(x)>C\},
\end{equation*}
where $x_0\in \Omega$, $R>0$ and $C>0$ depends on parameters and the cut-off function $\eta$.  Consequently, applying the maximum principle on $\Omega'$, we deduce the estimates for $u^\alpha w$ on $\Omega$.

\begin{figure}
\centering
\begin{tikzpicture}
\begin{axis}[
    axis lines = middle,
    axis line style = {->,line width=0.98pt},
    axis equal,
    xlabel = {$p$},
    ylabel = {$q$},
    xmin = -1, xmax = 4,
    ymin = -1.5, ymax = 4,
    ticks = none,
    width = 9cm,
    height = 8cm,
    xlabel style={at={(axis description cs:1,0.26)},anchor=west},
    ylabel style={at={(axis description cs:0.26,1)},anchor=south},
]

\fill[blue!50] (axis cs:-1.7,-1.7) rectangle (axis cs:2,1.5);
\fill[orange!90] (axis cs:-1.7,5) rectangle (axis cs:0, 1.5);
\fill[yellow!50] (axis cs:0,1.5) rectangle (axis cs:2,5);
\fill[red!30]
    (axis cs:2,-1.7)
    -- (axis cs:2.0, 1.333)    
    -- (axis cs:2.2, 1.375)    
    -- (axis cs:2.4, 1.412)    
    -- (axis cs:2.6, 1.444)    
    -- (axis cs:2.8, 1.474)    
    -- (axis cs:3.0, 1.500)    
    -- (axis cs:3.2, 1.524)    
    -- (axis cs:3.4, 1.545)    
    -- (axis cs:3.6, 1.565)    
    -- (axis cs:3.8, 1.583)    
    -- (axis cs:4.0, 1.600)    
    -- (axis cs:4.1, 1.607)  
    -- (axis cs:4.2, 1.615)  
    -- (axis cs:4.3, 1.623)  
    -- (axis cs:4.4, 1.630)  
    -- (axis cs:4.5, 1.636)  
    -- (axis cs:4.6, 1.643)  
    -- (axis cs:4.7, 1.649)  
    -- (axis cs:4.7, 0)        
-- (axis cs:4.7,-1.7)
    -- cycle;

\fill[green!60]
    (axis cs:2,4.7)
   -- (axis cs:2.0, 1.333)    
    -- (axis cs:2.2, 1.375)    
    -- (axis cs:2.4, 1.412)    
    -- (axis cs:2.6, 1.444)    
    -- (axis cs:2.8, 1.474)    
    -- (axis cs:3.0, 1.500)    
    -- (axis cs:3.2, 1.524)    
    -- (axis cs:3.4, 1.545)    
    -- (axis cs:3.6, 1.565)    
    -- (axis cs:3.8, 1.583)    
    -- (axis cs:4.0, 1.600)    
    -- (axis cs:4.1, 1.607)  
    -- (axis cs:4.2, 1.615)  
    -- (axis cs:4.3, 1.623)  
    -- (axis cs:4.4, 1.630)  
    -- (axis cs:4.5, 1.636)  
    -- (axis cs:4.6, 1.643)  
    -- (axis cs:4.7, 1.649)  
    -- (axis cs:4.7, 4.7)
    -- cycle;

\draw[->, line width=1pt, black] (axis cs:{-1.7}, 0) -- (axis cs:{4.7}, 0);

\draw[->, line width=1pt, black] (axis cs:0, -1.5) -- (axis cs:{0}, 4);
\addplot[domain=-1.5:4, line width=0.2pt] coordinates {
    ({2}, -2)
    ({2}, 5)
};
\node at (axis cs:{2.23},-0.9) {$p=\frac{N+3}{N-1}$};

\addplot[domain=-3:4, line width=0.2pt] coordinates {
    ({-1.7}, 1.5)
    ({2}, 1.5)
};
\node at (axis cs:-0.9,{1.2}) {$q=\frac{N+1}{N-1}$};

\addplot[domain=-3:4, line width=0.1pt,dashed] coordinates {
    ({0}, 1)
    (1, 1)
};

\addplot[domain=-3:4, line width=0.1pt,dashed] coordinates {
    (1, 0)
    (1, 1)
};

\addplot[
    domain=1:4.7,
    samples=200,
    line width=0.8pt,
    densely dashed
] (x,{2*x/(x+1)});
\node at (axis cs:4.0,1.9) {$q=\frac{2p}{p+1}$};
\node at (axis cs:1.5,0.8) {$(1,1)$};
\node at (axis cs:0,0) [below left] {$0$};
\end{axis}
\end{tikzpicture}
\begin{tikzpicture}
\path[fill=red!0]
  (4,3)-- (5,3)-- (5,3.5)-- (4,3.5)-- cycle;
\draw[dashed, thick]
  (4,3.25)-- (5,3.25);
\node[anchor=west] at (5,3.25){$G_6$};
\path[fill=yellow!50](4, 4)--(5,4)--(5,4.2) node[anchor=west]{$G_5$}--(5,4.5)--(4,4.5);
\path[fill=orange!90](4, 5)--(5,5)--(5,5.2) node[anchor=west]{$G_4$}--(5,5.5)--(4,5.5);
\path[fill=blue!50](4, 6)--(5,6)--(5,6.2) node[anchor=west]{$G_3$}--(5,6.5)--(4,6.5);
\path[fill=green!60](4, 7)--(5,7)--(5,7.2) node[anchor=west]{$G_2$}--(5,7.5)--(4,7.5);
\path[fill=red!30](4, 8)--(5,8)--(5,8.2) node[anchor=west]{$G_1$}--(5,8.5)--(4,8.5);
\end{tikzpicture}	
\caption{The range of $(p,q)$ when $N=5$.}\label{fig:parameter_space}
\end{figure}
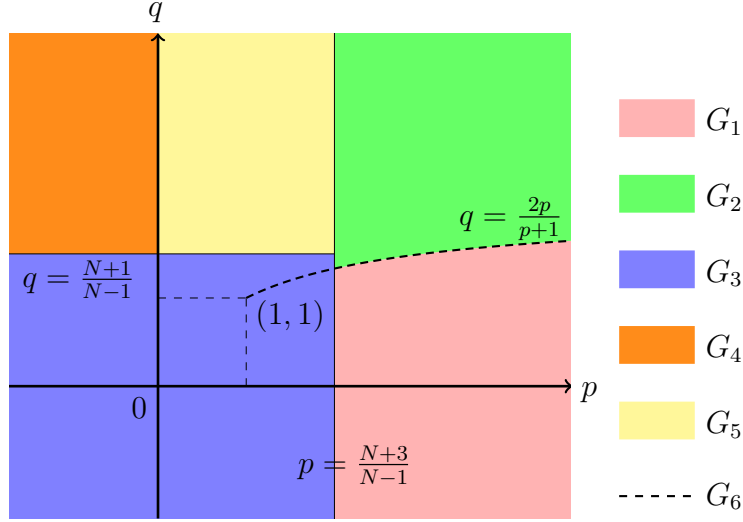

We now turn to the supercritical case of $q$ in the following theorem.

\begin{theorem}\label{them:p>,q>}
Let $(p,q)\in G_2$. Assume $u$ is a positive solution of \eqref{eq1} in $\Omega$ satisfying
\begin{equation}\label{u>M2}
  u\geq c_{N,p,q}M^{\frac{2}{2p-(p+1)q}}
\end{equation}
for some $c_{N,p,q}>0$. Then there exists $C=C(N,M,q)>0$ such that
\begin{equation*}
  |\nabla u(x)|\leq C{\rm dist}^{-\frac{1}{q-1}}(x,\partial \Omega),\quad x\in\Omega.
\end{equation*}
Consequently, \eqref{eq1} possesses no positive solution in $\mathbb R^N$ satisfying \eqref{u>M2}.
\end{theorem}

When $p$ and $q$ are smaller than those in the preceding cases, including negative exponents, we establish the following result for positive solutions to \eqref{eq1}.

\begin{theorem}\label{them:p,q<n+3,n+1}
Let $(p,q)\in G_3$. Assume $u$ is a positive solution of \eqref{eq1} in $\Omega$, then there exists $C=C(N,p,q)>0$ such that
\begin{equation*}
  \frac{|\nabla u(x)|}{u(x)}\leq C{\rm dist}^{-1}(x,\partial \Omega),\quad x\in\Omega.
\end{equation*}
Consequently, \eqref{eq1} possesses no positive solution in $\mathbb R^N$.
\end{theorem}

The proof of this result relies on the inequality
\begin{equation*}
\mathcal L(w^\gamma):= \mathcal L_0(w^\gamma)\leq -w^{\gamma+1}
\end{equation*}
with $\gamma>0$ large enough. The choice of $w^\gamma$ yields the negative term $-\gamma(\gamma-1)w^{-1}|\nabla w|^2$, which absorbs positive terms arising from expansion of $-\gamma |D^2v|^2$. By introducing a cut-off function, we derive the pointwise gradient estimates via the maximum principle.

In Theorems \ref{theorem:p>N+3,q<1} and \ref{them:p,q<n+3,n+1}, we have established the gradient estimates and Liouville-type properties of solutions to \eqref{eq1} when $p\geq (N+3)/(N-1)$ or $q<(N+1)/(N-1)$. We now divide the remaining region of the plane into two subregions, denoted by $G_4$ and $G_5$. In the region $G_4$,  where $p\leq 0$, the nonexistence of positive solutions to \eqref{eq1} follows as a special case of the result obtained by Bai, Zhang and Zhang \cite[Theorem 1.1]{Bai-Zhang} via pointwise gradient estimates. We note that $u^p$ has no effect on gradient estimates in this region. Indeed, selecting $f=-Id$, where $Id$ denotes the identity operator, we obtain that $\mathcal L(|\nabla u|^2)\leq 2p u^{p-1}|\nabla u|^2-2|D^2u|^2<0$. Replacing $|\nabla u|^2$ by $|\nabla u|^2\eta$, we arrive at the following gradient estimates.

\begin{theorem}\label{them:p<=0,q>}
Let $(p,q)\in G_4$. Assume $u$ is a positive solution of \eqref{eq1} in $\Omega$, then there exists $C=C(N,M,q)>0$ such that
\begin{equation*}
  |\nabla u(x)|\leq C{\rm dist}^{-\frac{1}{q-1}}(x,\partial \Omega),\quad x\in\Omega.
\end{equation*}
Consequently,  \eqref{eq1} possesses no positive solution in $\mathbb R^N$.
\end{theorem}

We obtain the subsequent theorem via delicate analysis of $\mathcal L(w^\gamma)$ with $\gamma>1$ on $G_5$, where $p>0$.

\begin{theorem}\label{them:0<p<,q>}
Let $(p,q)\in G_5$. Assume $u$ is a positive solution of \eqref{eq1} in $\Omega$ satisfying
\begin{equation}\label{u<M3}
  u\leq c_{N,p,q}M^{\frac{2}{2p-(p+1)q}} \quad \text{for}\ q\leq2,
\end{equation}
or
\begin{equation}\label{u>M3}
  u\geq c_{N,p,q}M^{\frac{2}{2p-(p+1)q}} \quad \text{for}\ q>2
\end{equation}
for some $c_{N,p,q}>0$. Then there exists $C=C(N,p,q)>0$ such that
\begin{equation*}
  \frac{|\nabla u(x)|}{u(x)}\leq C{\rm dist}^{-1}(x,\partial \Omega),\quad x\in\Omega.
\end{equation*}
Consequently, \eqref{eq1} possesses no positive solution in $\mathbb R^N$ satisfying conditions  \eqref{u<M3} or \eqref{u>M3}.
\end{theorem}

In $G_5$, we impose two distinct boundedness assumptions on $u$ to derive the Liouville-type result. The reason is that $w^{(q+2)/2}$ coming from $|\nabla u|^q$ can be absorbed by $w^2$ or $w^q$ arising from $(\Delta u)^2$.

In $G_6$, where $q=2p/(p+1)$ with $p>1$, which contains the common boundary of $G_1$ and $G_2$, solutions of \eqref{eq1} admit different  behaviors from that exhibited in $G_1$ and $G_2$. More precise, the coefficient $M$ plays a fundamental role rather than the boundedness restriction on $u$. Subsequently, we establish the estimates of solutions.

\begin{theorem}\label{them:uni-esti}
Let $(p,q)\in G_6$. Assume $u$ is a positive solution of \eqref{eq1} in $\Omega$ for
\begin{equation}\label{defi-M0}
  M\geq M_0:=2^{-\frac{p-1}{p+1}}(6N)^{\frac{p}{p+1}}(N+2)^{\frac{p-1}{p+1}}(p-1)^{\frac{p-1}{p+1}}p^{-\frac{p}{p+1}}(p+1).
\end{equation}

{\rm (i)} There exists a constant $C=C(N,M,p)>0$ such that
\begin{equation}\label{uni-esti}
  u(x)+|\nabla u(x)|^{\frac2{p+1}}\leq C{\rm dist}^{-\frac2{p-1}}(x,\partial\Omega),\quad x\in\Omega.
\end{equation}
Consequently, \eqref{eq1} possesses no positive solution in $\mathbb R^N$.

{\rm (ii)} Suppose $u\leq m$ for some $m>0$, then there exists a constant $C=C(N,M,p)>0$ such that
the solution satisfies
\begin{equation}\label{u<m-esti}
  |\nabla u(x)|\leq Cm{\rm dist}^{-1}(x,\partial\Omega),\quad x\in\Omega.
\end{equation}
\end{theorem}

Estimates \eqref{uni-esti} is called universal in the sense that the constant $C$ does not depend on solutions and the domain $\Omega$. Following ideas from \cite {Polacik-Quittner-Souplet,Souplet-DCDS}, we shall use Liouville-type theorems to obtain universal estimates. To this end, we first prove the nonexistence of bounded solutions to \eqref{eq1} by local gradient estimates in Corollary \ref{corol-u-bound}. Next, using the doubling arguments and limiting procedures, we derive the universal estimates \eqref{uni-esti}.

It is remarkable that estimates \eqref{uni-esti} is a complement of that obtained in \cite[Theorem C]{Veron-MathAnn-20}. Actually, for $x$ in a neighborhood of $\partial \Omega$, there exists $x_0\in\partial\Omega$ such that $\delta(x):={\rm dist}(x,\partial\Omega)=|x-x_0|$. We denote $\nu(x_0)$ the inward unit normal vector at $x_0\in\partial\Omega$. Then
\begin{equation*}
  u(x)-u(x_0)=\int_0^{\delta(x)}\nabla u(x_0+s\nu(x_0))\cdot\nu(x_0){\rm d}s\leq C\int_0^{\delta(x)}s^{-\frac{p+1}{p-1}}{\rm d}s,
\end{equation*}
where $C>0$ and the integral on the right-hand side of the inequality is an improper integral due to $(p+1)/(p-1)>1$. Additionally, we point out that, concerning the gradient blow-up behavior of solutions in a bounded domain, estimates \eqref{u<m-esti} is more precise than \eqref{uni-esti}.

Finally, we intend to establish an universal estimates for the following superlinear elliptic equation with general nonlinearities
\begin{equation}\label{eq-f+g}
  -\Delta u=f(u)+g(\nabla u),
\end{equation}
where $f$ and $g$ are continuous functions satisfying suitable conditions without scale invariance.

\begin{theorem}\label{them:f-g}
Let $\Omega$ be an arbitrary domain of $\mathbb R^N$. Assume $f\in C([0,\infty))$ and  $g\in C(\mathbb R^N)$ satisfy
\begin{equation}\label{assum-f}
  \lim_{s\to \infty}s^{-p}f(s)=l>0,
\end{equation}
and
\begin{equation}\label{assum-g}
  \quad \lim_{s\to \infty}|s\xi|^{-\frac{2p}{p+1}}g\left(s\xi\right)=\mu>0\quad \text{for}\ \xi\neq0,
\end{equation}
respectively.

{\rm (i)} Suppose $1<p<{(N+3)/(N-1)}$. Then there exists $C=C(N,p,f,g)>0$ such that the nonnegative solution to \eqref{eq-f+g} in $\Omega$ satisfies
\begin{equation}\label{uni-f+g}
u(x)+|\nabla u(x)|^{\frac{2}{p+1}}\leq C\left(1+{\rm dist}^{-\frac2{p-1}}(x,\partial\Omega)\right),\quad x\in\Omega.
\end{equation}

{\rm (ii)} Suppose $p\geq (N+3)/(N-1)$. If $\mu l^{-1/(p+1)}\geq M_0$, where $M_0$ is defined in \eqref{defi-M0},  then the nonnegative solution of \eqref{eq-f+g} in $\Omega$ satisfies \eqref{uni-f+g}.
\end{theorem}

We note that Theorem \ref{them:f-g} completes the result in \cite[Theorem 6.1]{Polacik-Quittner-Souplet}, where universal estimates were obtained  when $g(\nabla u)$ has subcritical power growth and $p\in (1,p_S)$, while our results occur in the critical case with $p>1$. Typical examples of $f$ and $g$ satisfying requirements in Theorem \ref{them:f-g} are given by $f(u)=\sum_{i=1}^ja_iu^{p_i}+lu^p$ and $g(\nabla u)=\sum_{i=1}^kb_i|\nabla u|^{q_i}+\mu|\nabla u|^{2p/(p+1)}$, respectively, where $j,k\in \mathbb R^+$, $a_i,b_i\in\mathbb R$, $0\leq p_i<p(p>1)$, and $0\leq q_i<2p/(p+1)$.

The paper is organized as follows. In Section \ref{sec:2}, we establish differential inequalities for solutions to \eqref{eq1} and the maximum principle. In Section \ref{sec:aux-funt-&-est}, we provide local gradient estimates by skillfully choosing auxiliary functions and then derive  Liouville-type theorems. In Section \ref{sec:uni-est}, as applications of Liouville-type theorems, we use doubling arguments to achieve the universal estimates for elliptic equations with general nonlinearities.

\section{Elliptic differential operators and inequalities}\label{sec:2}

In this section, we establish differential inequalities involving undetermined auxiliary functions for solutions to \eqref{eq1}, which play a key role in deriving the local gradient estimates.

Let $a\in(0,1)$ be chosen later, $x_0\in\Omega$ be fixed, and  $R={\rm dist}(x_0,\partial\Omega)$. Set $R^{\prime}=3 R/4$. We select a cut-off function $\eta \in C^2\left(\overline{B}_{R}\right)$, $0 \leq \eta \leq 1$, satisfying $\eta=1$ for $|x-x_0|\leq R/2$, $\eta=0$ for $\left|x-x_0\right|\geq R^{\prime}$ and such that
\begin{equation}\label{dfi-eta}
\left.\begin{array}{rl}
|\nabla \eta| & \leq C R^{-1} \eta^a \\
\left|D^2 \eta\right|+\eta^{-1}|\nabla \eta|^2 & \leq C R^{-2} \eta^a
\end{array}\right\} \text { for }\left|x-x_0\right|<R^{\prime},
\end{equation}
where $C=C(a)>0$ and $B_R=B_R(x_0)$. Indeed, such a function is given in \cite{Souplet-Zhang} as $\eta=\rho^k$, where $\rho(x)=1-R'^{-2}|x-x_0|^2$, $x\in B_R$ and $k\geq2/(1-a)$.

Let $f$ be a $C^3$ function with $f'\neq0$. We set
\begin{equation}\label{def-v}
v=f^{-1}(-u),\quad w=|\nabla v|^2.
\end{equation}
By a direct computation, $v$ satisfies
\begin{equation}\label{v-w}
  -\Delta v=-\frac{(-f)^p}{f'}-M\frac{|f'|^q}{f'}w^{\frac q2}+\frac{f''}{f'}w.
\end{equation}
For convenience, the variable $v$ of $f$, $f'$ and $f''$  is omitted here and hereafter.  By the elliptic regularity theory, we know that the solution $u$ is actually smooth away from $\Omega_0$ and hence we can differentiate the equation.
Then for $w>0$, due to Bochner's identity
\begin{equation*}
  \Delta w=2\langle\nabla \Delta v,\nabla v\rangle+2|D^2v|^2,
\end{equation*}
we obtain
\begin{align}\label{eq_w}\nonumber
-\Delta w=&2p(-f)^{p-1}w+2\frac{f''}{(f')^2}(-f)^pw+2\left(\frac{f''}{f'}\right)'w^2+2\frac{f''}{f'}\langle\nabla w,\nabla v\rangle\\
&-2M(q-1)|f'|^{q-2}f''w^{\frac{q+2}{2}}-M q\frac{|f'|^q}{f'}w^{\frac{q-2}{2}}\langle \nabla w,\nabla v\rangle-2|D^2v|^2,
\end{align}
where $|D^2 v|^2=\sum_{i,j=1}^N(v_{ij})^2$. We consider the operator $\mathcal L_\alpha$  defined as
\begin{equation}\label{Lalpha-2}
  \mathcal L_\alpha(z)=-\Delta z+\mathcal H_\alpha\cdot\nabla z,
\end{equation}
where
\begin{equation}\label{defin-H}
\mathcal H_\alpha= \left[qM\frac{|f'|^q}{f'}w^{\frac{q-2}{2}}-2\frac{f''}{f'}+2\alpha\frac{f'}f\right]\nabla v.
\end{equation}
Setting $\alpha=0$ in \eqref{Lalpha-2}, then \eqref{eq_w} can be written as
\begin{equation*}
\mathcal L(w):=\mathcal L_0(w)=-2|D^2v|^2+\mathcal N(w),
\end{equation*}
where
\begin{equation*}
  \mathcal N(w):=2p(-f)^{p-1}w+2\frac{f''}{(f')^2}(-f)^pw+2\left(\frac{f''}{f'}\right)'w^2-2(q-1)M|f'|^{q-2}f''w^{\frac{q+2}2}.
\end{equation*}

Now we present a local differential inequality of $u^\alpha w^\gamma$, which will be used frequently throughout this article.

\begin{lemma}\label{lem:L(wueta)}
Assume that $u$ is a positive solution of \eqref{eq1} in $\Omega$. Then for any $\alpha\in\mathbb R$ and $\gamma\geq 1$,
\begin{align}\label{ineq-alpha-gamma}\nonumber
& u^{1-\alpha}w^{1-\gamma}\mathcal L_\alpha(u^\alpha w^\gamma\eta)\\\nonumber
\leq&\gamma u\left[-(\gamma-1)w^{-1}|\nabla w|^2-2|D^2v|^2+\left(2p+2\frac{f''u}{(f')^2}+\frac{\alpha}\gamma\right)u^{p-1}w+2\left(\frac{f''}{f'}\right)'w^2\right]\eta\\
&-\gamma M(q-1)\left(2 |f'|^{q-2}f''u+\frac{\alpha}{\gamma}|f'|^q\right)w^{\frac{q+2}2}\eta+\left(\alpha(\alpha+1)u^{-1}(f')^2+2\alpha f''\right)w^2\eta\\\nonumber
&-2\gamma u\langle \nabla w,\nabla\eta\rangle +uw\left(|q|M|f'|^{q-1}w^{\frac{q-1}2}+2\left|\frac{f''}{f'}\right|w^{\frac12}\right)|\nabla\eta|+\sqrt N uw|D^2\eta|
\end{align}
holds in $\{x\in B_{R'};\,w(x)>0\}$.
\end{lemma}

\begin{proof}
By a direct computation,  we follow from \eqref{def-v} that
\begin{align}\label{Delta_u_alpha}\nonumber
  -\Delta(u^\alpha)=&-\alpha(\alpha-1)u^{\alpha-2}|\nabla u|^2+\alpha u^{\alpha-1}\left(u^p+M|\nabla u|^q\right)\\
  =&-\alpha(\alpha-1)u^{\alpha-2}(f')^2w+\alpha u^{\alpha-1}\left(u^p+M|f'|^qw^{\frac{q}2}\right).
\end{align}
From \eqref{eq_w},  we get
\begin{align}\label{Delta_w_gam}\nonumber
  &-\Delta(w^\gamma)\\
 =&-\gamma(\gamma-1)w^{\gamma-2}|\nabla w|^2-\gamma w^{\gamma-1}\Delta w\\\nonumber
 =&-\gamma(\gamma-1)w^{\gamma-2}|\nabla w|^2-2\gamma w^{\gamma-1}|D^2v|^2+\gamma w^{\gamma-1}\bigg(2p(-f)^{p-1}w+2\frac{f''}{(f')^2}(-f)^pw\\\nonumber
&+2\left(\frac{f''}{f'}\right)'w^2+\!2\frac{f''}{f'}\langle\nabla w,\nabla v\rangle-\!2 M(q-1)|f'|^{q-2}f''w^{\frac{q+2}{2}}-qM\frac{|f'|^q}{f'}w^{\frac{q-2}{2}}\langle\nabla w,\nabla v\rangle\bigg).
\end{align}
As for $u^\alpha w^\gamma\eta$, a calculation shows that
\begin{align}\label{ualpha-wgamma}
  &-\Delta(u^\alpha w^{\gamma}\eta)\\\nonumber
  =&-u^{\alpha}\Delta(w^\gamma)\eta-w^\gamma\Delta(u^\alpha)\eta-u^\alpha w^\gamma\Delta\eta-2\langle \nabla u^\alpha,\nabla w^\gamma\rangle\eta-2\langle\nabla(u^\alpha w^\gamma),\nabla\eta\rangle.
\end{align}
Substituting \eqref{Delta_u_alpha} and \eqref{Delta_w_gam} into \eqref{ualpha-wgamma}, we obtain
\begin{align}\label{gamma_u_alpha}\nonumber
  &-u^{1-\alpha}w^{1-\gamma}\Delta(u^\alpha w^\gamma\eta)\\\nonumber
  =&\gamma u\bigg[-(\gamma-1)w^{-1}|\nabla w|^2-2|D^2v|^2+2pu^{p-1}w+2\frac{f''}{(f')^2}u^pw+2\left(\frac{f''}{f'}\right)'w^2\\\nonumber
  &-2M(q-1)|f'|^{q-2}f''w^{\frac{q+2}2}\bigg]\eta+ w\left(-\alpha(\alpha-1)u^{-1}(f')^2w+\alpha u^p+\alpha M|f'|^qw^{\frac{q}2}\right)\eta\\
  &+ \left(-qM\frac{|f'|^q}{f'}w^{\frac{q-2}2}+2\frac{f''}{f'}\right) uw^{1-\gamma}\eta\langle\nabla w^\gamma,\nabla v\rangle-uw\Delta\eta\\\nonumber
  &-2u^{1-\alpha}w^{1-\gamma}\eta\langle\nabla u^{\alpha},\nabla w^{\gamma}\rangle-2 u^{1-\alpha}w^{1-\gamma}\langle\nabla(u^\alpha w^\gamma),\nabla\eta\rangle.
\end{align}

Recalling the definition of $\mathcal H_\alpha$, we have
\begin{align}\label{H-alpha-expres}\nonumber
  &u^{1-\alpha}w^{1-\gamma}\mathcal H_\alpha\cdot \nabla (u^\alpha w^\gamma \eta)\\
  =& u^{1-\alpha}w^{1-\gamma}\left(qM\frac{|f'|^q}{f'} w^{\frac{q-2}2}-2\frac{f''}{f'}+2\alpha \frac{f'}f\right)\nabla v\cdot\nabla(u^\alpha w^\gamma\eta)\\\nonumber
  =&\left(qM\frac{|f'|^q}{f'} w^{\frac{q-2}2}-2\frac{f''}{f'}\right)\Big[u^{1-\alpha}w\eta\langle\nabla v,\nabla u^\alpha\rangle+uw^{1-\gamma}\eta\langle\nabla v,\nabla w^\gamma\rangle+uw\langle\nabla v,\nabla \eta\rangle\Big]\\\nonumber
  &+2\alpha \frac{f'}f u^{1-\alpha} w^{1-\gamma}\langle\nabla v,\nabla(u^\alpha w^\gamma\eta)\rangle.
\end{align}
From $\nabla u=-f'\nabla v$, we obtain
\begin{align*}
  \langle\nabla v,\nabla u^\alpha\rangle=\alpha u^{\alpha-1}\langle\nabla v,\nabla u\rangle=-\alpha u^{\alpha-1}f'w,
\end{align*}
and hence
\begin{equation}\label{nabla_ualpha1*}
  u^{1-\alpha}w\eta\left(qM\frac{|f'|^q}{f'} w^{\frac{q-2}{2}}-2\frac{f''}{f'}\right)\langle\nabla v,\nabla u^\alpha\rangle=-\alpha qM|f'|^qw^{\frac{q+2}2}\eta+2\alpha f''w^2\eta.
\end{equation}
After some computations we get
\begin{align*}
2\langle\nabla u^\alpha,\nabla w^\gamma\rangle\eta=&2\alpha\frac{f'}{f}\langle\nabla v,u^\alpha\eta\nabla w^\gamma\rangle\\
=&2\alpha\frac{f'}f\big\langle\nabla v,\nabla(u^\alpha w^\gamma\eta)-w^\gamma\eta\nabla u^\alpha-u^\alpha w^\gamma\nabla\eta\big\rangle\\\nonumber
=&2\alpha\frac{f'}f\langle\nabla v,\nabla(u^\alpha w^\gamma\eta)\rangle+2\alpha^2\frac{(f')^2}{f}u^{\alpha-1}w^{\gamma+1}\eta-2\alpha\frac{f'}fu^\alpha w^\gamma\langle\nabla v,\nabla\eta\rangle,
\end{align*}
which means
\begin{align}\label{**2alp-ff}\nonumber
  &2\alpha\frac{f'}fu^{1-\alpha}w^{1-\gamma}\langle\nabla v,\nabla(u^\alpha w^\gamma\eta)\rangle\\
  =&2u^{1-\alpha}w^{1-\gamma}\langle\nabla u^\alpha,\nabla w^\gamma\rangle\eta-2\alpha^2\frac{(f')^2}{f}w^{2}\eta+2\alpha\frac{f'}fu w\langle\nabla v,\nabla\eta\rangle.
\end{align}
Substituting \eqref{nabla_ualpha1*} and \eqref{**2alp-ff} into \eqref{H-alpha-expres}, it leads to
\begin{align}\label{uwH-alpha}\nonumber
 &u^{1-\alpha}w^{1-\gamma}\mathcal H_\alpha\cdot \nabla (u^\alpha w^\gamma \eta)\\
 =&-\alpha qM|f'|^qw^{\frac{q+2}2}\eta+2\alpha f''w^2\eta+2u^{1-\alpha}w^{1-\gamma}\eta\langle\nabla u^\alpha,\nabla w^\gamma\rangle-2\alpha^2\frac{(f')^2}{f}w^{2}\eta\\\nonumber
 &+2\alpha\frac{f'}fu w\langle\nabla v,\nabla\eta\rangle+\left(qM\frac{|f'|^q}{f'} w^{\frac{q-2}2}-2\frac{f''}{f'}\right)\Big[uw^{1-\gamma}\eta\langle\nabla v,\nabla w^\gamma\rangle+uw\langle\nabla v,\nabla \eta\rangle\Big]
\end{align}
It follows from \eqref{Lalpha-2}, \eqref{gamma_u_alpha} and \eqref{uwH-alpha} that
\begin{align}\label{L-u-alpha-2}\nonumber
&u^{1-\alpha}w^{1-\gamma}\mathcal L_\alpha(u^\alpha w^\gamma\eta)\\\nonumber
=&\gamma u\left[-(\gamma-1)w^{-1}|\nabla w|^2-2|D^2v|^2+\left(2p+2\frac{f''u}{(f')^2}+\frac{\alpha}\gamma\right)u^{p-1}w+2 \left(\frac{f''}{f'}\right)'w^2\right]\eta\\
&-\gamma M(q-1)\left(2 |f'|^{q-2}f''u+\frac{\alpha}{\gamma}|f'|^q\right)w^{\frac{q+2}2}\eta-uw\Delta\eta\\\nonumber
&+\left(-\alpha(\alpha-1)u^{-1}(f')^2+2\alpha f''-2\alpha^2\frac{(f')^2}f\right)w^2\eta+2\alpha\frac{f'}fu w\langle\nabla v,\nabla\eta\rangle\\\nonumber
&-2u^{1-\alpha}w^{1-\gamma}\langle\nabla(u^\alpha w^\gamma),\nabla\eta\rangle+uw\Big(qM\frac{|f'|^q}{f'}w^{\frac{q-2}2}-2\frac{f''}{f'}\Big)\langle\nabla v,\nabla\eta\rangle.
\end{align}

We now consider the terms containing $\nabla\eta$ and $\Delta\eta$. We notice that
\begin{align}\label{nabla-v,nabla-eta}\nonumber
-2\langle\nabla(u^\alpha w^\gamma),\nabla\eta\rangle=&-2w^\gamma\langle\nabla u^\alpha,\nabla\eta\rangle-2u^\alpha \langle\nabla w^\gamma,\nabla\eta\rangle\\
=&2\alpha u^{\alpha-1}w^\gamma f'\langle\nabla v,\nabla\eta\rangle-2\gamma u^\alpha w^{\gamma-1}\langle\nabla w,\nabla\eta\rangle\\\nonumber
=&-2\alpha \frac{f'}{f}u^\alpha w^\gamma\langle\nabla v,\nabla\eta\rangle-2\gamma u^\alpha w^{\gamma-1}\langle\nabla w,\nabla\eta\rangle.
\end{align}
From Cauchy--Schwarz's inequality, we have
\begin{equation}\label{nabla-eta-1}
\left(qM\frac{|f'|^q}{f'}w^{\frac{q-2}2}-2\frac{f''}{f'}\right)\langle\nabla v,\nabla\eta\rangle\leq \left(|q|M|f'|^{q-1}w^{\frac{q-1}2}+2\left|\frac{f''}{f'}\right|w^{\frac12}\right)|\nabla\eta|.
\end{equation}
In addition, there holds
\begin{align}\label{1/N}
-\Delta\eta\leq \sqrt N|D^2\eta|.
\end{align}
Substituting \eqref{nabla-v,nabla-eta}--\eqref{1/N} into \eqref{L-u-alpha-2}, we obtain
\begin{align*}\nonumber
& u^{1-\alpha}w^{1-\gamma}\mathcal L_\alpha(u^\alpha w^\gamma\eta)\\\nonumber
\leq&\gamma u\left[-(\gamma-1)w^{-1}|\nabla w|^2-2|D^2v|^2+\left(2p+2\frac{f''u}{(f')^2}+\frac{\alpha}\gamma\right)u^{p-1}w+2\left(\frac{f''}{f'}\right)'w^2\right]\eta\\
&-\gamma M(q-1)\left(2 |f'|^{q-2}f''u+\frac{\alpha}{\gamma}|f'|^q\right)w^{\frac{q+2}2}\eta+\left(\alpha(\alpha+1)u^{-1}(f')^2+2\alpha f''\right)w^2\eta\\\nonumber
&-2\gamma u\langle \nabla w,\nabla\eta\rangle +uw\left(|q|M|f'|^{q-1}w^{\frac{q-1}2}+2\left|\frac{f''}{f'}\right|w^{\frac12}\right)|\nabla\eta|+\sqrt N uw|D^2\eta|.
\end{align*}
Thus, we complete the proof of this lemma.
\end{proof}

Furthermore, we present the estimates of $\mathcal L(w^\gamma\eta):=\mathcal L_0(w^\gamma\eta)$ under a specific local orthonormal frame.

\begin{lemma}\label{lem:e1}
Assume that $u$ is a positive solution of \eqref{eq1} in $\Omega$. Then for any $\gamma\geq1$, there exists a constant $C=C(N,q,\gamma)>0$ such that
\begin{align*}
  &w^{1-\gamma}\mathcal L(w^{\gamma}\eta)\\
  \leq &\gamma\left(-2(\gamma-1)+\frac{2(1-N+3\gamma^{1/2})}{N-1}\right)v^2_{11}\eta+2\gamma\left(p+\frac{N+1}{N-1}\frac{f''u}{(f')^2}\right)u^{p-1}w\eta\\
  &+2\gamma \left(\frac{f''}{f'}\right)'w^2\eta-2\gamma M\left(q-\frac{N+1}{N-1}\right)|f'|^{q-2}f''w^{\frac{q+2}2}\eta-\frac{\gamma (1-2\gamma^{-1/2})}{N-1}M^2|f'|^{2q-2}w^q\eta\\
  &+\bigg[-\frac{2\gamma(1-\gamma^{-1/2})}{N-1}\left(\frac{f''}{f'}\right)^2+3\bigg]w^2\eta-\frac{2\gamma(1-\gamma^{-1/2})}{N-1}\frac{(-f)^{2p}}{(f')^2}\eta\\
  &+2\left|\frac{f''}{f'}\right|w^{\frac32}|\nabla\eta|+C\left(\eta^{-3}|\nabla\eta|^4+\eta^{-1}|D^2\eta|^2\right)
\end{align*}
holds in $\{x\in B_{R'};\,w(x)>0\}$.
\end{lemma}

\begin{proof}
Taking $\alpha=0$ in Lemma \ref{lem:L(wueta)}, we obtain
\begin{align}\label{alpha=0}\nonumber
& w^{1-\gamma}\mathcal L(w^\gamma\eta)\\
\leq&\gamma\left(-(\gamma-1)w^{-1}|\nabla w|^2-2|D^2v|^2\right)\eta+2\gamma\left(p+\frac{f''u}{(f')^2}\right)u^{p-1}w\eta\\\nonumber
&+2\gamma\left(\frac{f''}{f'}\right)'w^2\eta-2\gamma M(q-1)|f'|^{q-2}f''w^{\frac{q+2}2}\eta-2\gamma \langle \nabla w,\nabla\eta\rangle \\\nonumber
& +w\left(|q|M|f'|^{q-1}w^{\frac{q-1}2}+2\left|\frac{f''}{f'}\right|w^{\frac12}\right)|\nabla\eta|+\sqrt N w|D^2\eta|.
\end{align}

Let $\{e_1,e_2,\cdots,e_N\}$ be a local orthonormal frame of the tangent bundle $T\mathbb R^N$ on a domain with $w>0$ such that $e_1=\nabla v/|\nabla v|$.  Setting $v_i=\partial v/\partial e_i$  for $i\in\{1,2,\ldots,N\}$, then we have
\begin{equation*}
 v_1=\langle\nabla v,e_1\rangle=w^{\frac12},
\end{equation*}
and
\begin{equation}\label{v_11}
  v_{11}=\frac12 w^{-\frac12}w_1=\frac12w^{-1}\langle\nabla w,\nabla v\rangle.
\end{equation}
From \eqref{v_11} and Cauchy-Schwarz's inequality, we obtain
\begin{equation*}
  v_{11}^2=\frac14 w^{-2}\langle\nabla w,\nabla v\rangle^2\leq \frac14 w^{-1}|\nabla w|^2.
\end{equation*}
Hence for any $\gamma\geq 1$, we know
\begin{equation}\label{gamma-1}
  -(\gamma-1)w^{-1}|\nabla w|^2\leq -4(\gamma-1)v_{11}^2.
\end{equation}

As for the estimates of $|D^2v|^2$, we note that
\begin{equation}\label{D2v-1}
  -|D^2v|^2\leq-v_{11}^2-\sum_{i=2}^N v_{ii}^2\leq -v_{11}^2-\frac{1}{N-1}\left(\sum_{i=2}^Nv_{ii}\right)^2.
\end{equation}
Since
\begin{equation*}
  \left(\sum_{i=2}^Nv_{ii}\right)^2=\left(-\Delta v+v_{11}\right)^2,
\end{equation*}
it follows from \eqref{v-w} that
\begin{align}\label{sum-vii}\nonumber
&-\left(\sum_{i=2}^N v_{ii}\right)^2\\
=&-\left(-\frac{(-f)^p}{f'}-M\frac{|f'|^q}{f'}w^{\frac q2}+\frac{f''}{f'}w+v_{11}\right)^2\\\nonumber
=&-v_{11}^2-\frac{(-f)^{2p}}{(f')^2}-M^2|f'|^{2q-2}w^q-\left(\frac{f''}{f'}\right)^2w^2-2M(-f)^p|f'|^{q-2}w^{\frac q2}\\\nonumber
&+\frac{2(-f)^p f''}{(f')^2}w+2M|f'|^{q-2}f''w^{\frac{q+2}2}-2v_{11}\left(-\frac{(-f)^p}{f'}-M\frac{|f'|^q}{f'}w^{\frac q2}+\frac{f''}{f'}w\right).
\end{align}
By Young's inequality, we have
\begin{equation}\label{fix-1}
  2v_{11}\frac{(-f)^p}{f'}\leq \gamma^{-\frac12}\frac{(-f)^{2p}}{(f')^2}+\gamma^{\frac12}v_{11}^2,
\end{equation}
\begin{equation}\label{fix-2}
  2Mv_{11}w^{\frac{q}2}\frac{|f'|^q}{f'}\leq \gamma^{-\frac12}M^2|f'|^{2q-2}w^q+\gamma^{\frac12}v_{11}^2,
\end{equation}
and
\begin{equation}\label{fix-3}
 -2v_{11}w\frac{f''}{f'}\leq \gamma^{-\frac12}\left(\frac{f''}{f'}\right)^2w^2+\gamma^{\frac12}v_{11}^2.
\end{equation}
Substituting \eqref{sum-vii}--\eqref{fix-3} into \eqref{D2v-1}, we have
\begin{align*}
-|D^2v|^2\leq& \left(\frac{3\gamma^{1/2}}{N-1}-1\right)v_{11}^2-\frac{1-\gamma^{-1/2}}{N-1}\left(\frac{f''}{f'}\right)^2w^2-\frac{1-\gamma^{-1/2}}{N-1}\frac{(-f)^{2p}}{(f')^2}\\
  &-\frac{(1-\gamma^{-1/2})}{N-1}M^2|f'|^{2q-2}w^q+\frac{2}{N-1}\left(\frac{(-f)^pf''}{(f')^2}w+M|f'|^{q-2}f''w^{\frac{q+2}2}\right).
\end{align*}
For terms containing $|\nabla\eta|$ and $|D^2\eta|$,  by Cauchy-Schwarz's and Young's inequalities, we obtain
\begin{align}
-2\gamma \langle\nabla w,\nabla\eta\rangle\leq&\frac{\gamma(\gamma-1)}4 w^{-1}|\nabla w|^2\eta+C(\gamma)w\eta^{-1}|\nabla\eta|^2.
\end{align}
Moreover, we have
\begin{equation*}
  C(\gamma)w\eta^{-1}|\nabla\eta|^2\leq w^2\eta+C(\gamma)\eta^{-3}|\nabla\eta|^4,
\end{equation*}
\begin{equation*}
  \sqrt N w|D^2\eta|\leq w^2\eta+C(N)\eta^{-1}|D^2\eta|^2,
\end{equation*}
and
\begin{align*}
  |q|M|f'|^{q-1}w^{\frac{q+1}2}|\nabla\eta|&\leq \frac{\gamma^{1/2}}{N-1}M^2|f'|^{2q-2}w^q\eta+C_1(N,q,\gamma)w\eta^{-1}|\nabla\eta|^2\\
  &\leq  \frac{\gamma^{1/2}}{N-1}M^2|f'|^{2q-2}w^q\eta+ w^2\eta+C(N,q,\gamma)\eta^{-3}|\nabla\eta|^4.
\end{align*}
Combining above estimates with \eqref{gamma-1} and going back to \eqref{alpha=0}, we deduce that
\begin{align*}
&w^{1-\gamma}\mathcal L(w^{\gamma}\eta)\\
  \leq &\gamma\left(-2(\gamma-1)+\frac{2(1-N+3\gamma^{1/2})}{N-1}\right)v^2_{11}\eta+2\gamma\left[p+\left(1+\frac{2}{N-1}\right)\frac{f''u}{(f')^2}\right]u^{p-1}w\eta\\
  &+\!2\gamma \left(\frac{f''}{f'}\right)'w^2\eta-\!2\gamma M\left(q-\!\frac{N+1}{N-1}\right)|f'|^{q-2}f''w^{\frac{q+2}2}\eta-\!\frac{\gamma (1\!-\!2\gamma^{-1/2})}{N-1}M^2|f'|^{2q-2}w^q\eta\\
  &+\bigg[-\frac{2\gamma(1-\gamma^{-1/2})}{N-1}\left(\frac{f''}{f'}\right)^2+3\bigg]w^2\eta-\frac{2\gamma(1-\gamma^{-1/2})}{N-1}\frac{(-f)^{2p}}{(f')^2}\eta\\
  &+2\left|\frac{f''}{f'}\right|w^{\frac32}|\nabla\eta|+C(N,q,\gamma)\left(\eta^{-3}|\nabla\eta|^4+\eta^{-1}|D^2\eta|^2\right).
\end{align*}
Thus we finish the proof.
\end{proof}

We also need the following maximum principle.

\begin{lemma}\label{lem:maxi-pric}
Let $\Omega\subset\mathbb R^N$ be a domain, $\vec{b}\in \mathbb R^N$ and $A>0$. Set
\begin{equation*}
  \Omega_A^+=\{x\in\Omega;\ z(x)>A\}.
\end{equation*}
Assume that $z$ is positive continuous in $\Omega$ and $C^2$ on $\Omega_A^+$. If $z$ satisfies
\begin{equation}\label{z>A}
  -\Delta z+\vec{b}\cdot \nabla z<0\quad \text{in}\ \Omega_A^+,
\end{equation}
and
\begin{equation*}
\vec{b}\in L^{\infty}(\Omega_A^+),
\end{equation*}
then $z\leq A$ in $\Omega$.
\end{lemma}

\begin{proof}
Assume for contradiction that there exists $x_0\in \Omega$ such that $z(x_0)>A$. Set
\begin{equation*}
  \bar z(x)=z(x)-A>0,\quad x\in\Omega_A^+,
\end{equation*}
then
\begin{equation*}
  -\Delta\bar z+\vec{b}\cdot\nabla \bar z<0 \quad\text{in} \ \Omega_A^+.
\end{equation*}
However, $\bar z=0$ on $\partial \Omega_A^+$. The maximum principle of classical solutions leads to that $\bar z\leq 0$ in $\Omega_A^+$, which gives that contradiction
\begin{equation*}
0<z(x_0)-A=\bar z(x_0)\leq 0.
\end{equation*}
Hence $\Omega_A^+=\emptyset$ and $z\leq A$ in $\Omega$, and the proof is complete.
\end{proof}

\section{Auxiliary functions and gradient estimates}\label{sec:aux-funt-&-est}
As in past studies on gradient estimates, the proof is based on the elaborate Bernstein technique. Within this framework, we introduce suitable auxiliary functions $f$ depending on the range of $(p,q)$ given by
\begin{align}\label{f1}
f(s)=
\begin{cases}
m(s+1)^\beta-2m,\quad &\text{for}\  G_1\ \text{and}\ G_6,\\
-s,\quad &\text{for}\  G_2\ \text{and}\ G_4,\\
-e^s,\quad &\text{for}\  G_3\ \text{and}\ G_5,
\end{cases}
\end{align}
where $m,\beta>0$ and the domain of $f$ will be determined later.

For pointwise gradient estimates, the case $\nabla u=0$ is trivial. It therefore suffices to derive the estimates in $\Omega\backslash\Omega_0$. From now on, all  computations are carried out on the set $\{x\in B_{R'};\,w(x)>0\}$ unless otherwise stated, and hence we can use Lemmas \ref{lem:L(wueta)} or \ref{lem:e1}.

\subsection {The case $p\geq (N+3)/(N-1)$}

In this section,  we consider the cases where $q$ is subcritical and supercritical.

\subsubsection{The subcritical range of $q$ }

We first give the gradient estimates for local  bounded solutions when $q$ is subcritical with respect to $p$.
\begin{lemma}\label{lem:p>N+3,q<2p/(p+1)}
Let $(p,q)\in G_1$. Assume $u=-f(v)$ is a positive solution of \eqref{eq1} in $\Omega$ that satisfies
\begin{equation}\label{u<M1}
  u\leq c_{N,p,q}M^{\frac{2}{2p-(p+1)q}}
\end{equation}
for some $c_{N,p,q}>0$ and $f$ defined in \eqref{f1}.
Set $z=u^\alpha w\eta$. Then for
\begin{equation*}
\max\left\{-\frac1{3N},-\frac4{3-q}\right\} <\alpha<0,
\end{equation*}
there exists a constant $C=C(N,M,p,q)>0$ such that
\begin{align*}
&u^{2-\alpha}(f')^{-2}w^{-2}\eta^{-1}\mathcal L_\alpha(z)\\
\leq & \frac{\alpha(\alpha+1)}2+C\left(z^{-1}|D^2\eta|+z^{-\frac{3-q}2}\eta^{\frac{1-q}2}|\nabla\eta|+z^{-\frac12}\eta^{-\frac12}|\nabla\eta|+z^{-1}\eta^{-1}|\nabla\eta|^2\right)
\end{align*}
holds in $\{x\in B_{R'};\,z(x)>0\}$.
\end{lemma}

\begin{proof}
Taking $\gamma=1$ in Lemma \ref{lem:L(wueta)}, there holds
\begin{align}\label{L_gamma=0}\nonumber
& u^{1-\alpha}\mathcal L_\alpha(u^\alpha w\eta)\\
\leq& u\left[-2|D^2v|^2+\left(2p+2\frac{f''u}{(f')^2}+\alpha\right)u^{p-1}w+2\left(\frac{f''}{f'}\right)'w^2\right]\eta\\\nonumber
&-M(q-1)\left(2 |f'|^{q-2}f''u+\alpha|f'|^q\right)w^{\frac{q+2}2}\eta+\left(\alpha(\alpha+1)u^{-1}(f')^2+2\alpha f''\right)w^2\eta\\\nonumber
&-2 u\langle \nabla w,\nabla\eta\rangle +uw\left(M|q||f'|^{q-1}w^{\frac{q-1}2}+2\left|\frac{f''}{f'}\right|w^{\frac12}\right)|\nabla\eta|+\sqrt N uw|D^2\eta|.
\end{align}
We shall find the corresponding upper bounds for nonnegative terms on the right-hand side of \eqref{L_gamma=0}. The nonnegative terms would be absorbed by $-|D^2v|^2\eta$ and $(f''/f')'w^2\eta$ with a suitable function $f$.

By \eqref{1/N} and \eqref{v-w}, we derive
\begin{align}\label{exp-D2v}\nonumber
  -|D^2v|^2\eta\leq&-\frac1N(\Delta v)^2\eta\\\nonumber
  =&-\frac1N\left(\frac{(-f)^p}{f'}+M\frac{|f'|^q}{f'}w^{\frac{q}2}-\frac{f''}{f'}w\right)^2\eta\\
  =&-\frac1N\bigg[\frac{(-f)^{2p}}{(f')^2}+M^2(f')^{2q-2}w^q+2M(-f)^p|f'|^{q-2}w^{\frac{q}2}\\\nonumber
  &+\left(\frac{f''}{f'}\right)^2w^2-2\frac{(-f)^pf''}{(f')^2}w-2M|f'|^{q-2}f''w^{\frac{q+2}{2}}\bigg]\eta.
\end{align}
From Young's inequality, we have
\begin{equation}\label{nablaw.nablaeta}
  2u\langle\nabla w,\nabla\eta\rangle\leq 2u|\nabla w||\nabla\eta|\leq 4u|D^2v|w^{\frac12}|\nabla\eta|\leq u\left(|D^2v|^2\eta+4w\eta^{-1}|\nabla \eta|^2\right).
\end{equation}
Subsitituting \eqref{exp-D2v} and \eqref{nablaw.nablaeta} into \eqref{L_gamma=0} leads to
\begin{align}\label{alpha<0}\nonumber
&u^{-\alpha}\mathcal L_\alpha(u^\alpha w\eta)\\\nonumber
\leq &\left(2p+\frac{2N+2}{N}\frac{f''}{(f')^2}u+\alpha\right)u^{p-1}w\eta-\frac{M^2}N(f')^{2q-2}w^q\eta-\frac1N\frac{(-f)^{2p}}{(f')^2}\eta\\
&-M\left[2\left(q-\frac{N+1}N\right)|f'|^{q-2}f''+\alpha(q-1)|f'|^qu^{-1}\right]w^{\frac{q+2}2}\eta\\\nonumber
&-\frac{1}{N}\left(\frac{f''}{f'}\right)^2w^2\eta+2\left(\frac{f''}{f'}\right)'w^2\eta+\left(\alpha(\alpha+1)(f')^2 u^{-2}+2\alpha f''u^{-1}\right)w^2\eta\\\nonumber
&+w\left(\!M|q||f'|^{q-1}w^{\frac{q-1}2}+\!2\left|\frac{f''}{f'}\right|w^{\frac12}+4\eta^{-1}|\nabla\eta|\right)|\nabla\eta|+\sqrt Nw|D^2\eta|,
\end{align}
where we omit the non-positive term $-2MN^{-1}(-f)^p|f'|^{q-2}u^{-1}w^{q/2}\eta$. Setting
\begin{equation*}
  m=\max_{x\in\overline {B}_{R'}}u(x),
\end{equation*}
and  following from \eqref{f1}, we note that $f$ maps $[0,2^{1/\beta}-1)$ into $[-m,0)$, and $f',f''>0$. Moreover, for $\beta >1$ to be determined later, we have
\begin{equation}\label{f''f-2u}
  \frac{f''}{(f')^2}u\leq \frac{\beta-1}{\beta}<1.
\end{equation}
Choosing $\alpha<0$, we also omit non-positive terms $\alpha u^{p-1}w\eta$ and $2\alpha f''u^{-1}w^2\eta$ in \eqref{alpha<0}.

Now we turn to estimates for $u^{p-1}w\eta$. Since $p\geq(N+3)/(N-1)$, it follows from $\eqref{f''f-2u}$ that
\begin{equation*}
\left(2p+\frac{2N+2}{N}\frac{f''}{(f')^2}u\right)u^{p-1}w\eta\leq 6pu^{p-1}w\eta.
\end{equation*}
For $q<1$, by Young's inequality with the exponent pair $2-q$, $(2-q)/(1-q)$, one has
\begin{align}\label{6pup-w-eta}\nonumber
 &6pu^{p-1}w\eta\\\nonumber
 \leq &\frac1{2-q}\left[\left(\frac{f''}{f'}w\right)^{-\frac{2(1-q)}{2-q}}u^{p-1}w\right]^{2-q}\eta+\frac{1-q}{2-q}\left[6p\left(\frac{f''}{f'}w\right)^{\frac{2(1-q)}{2-q}}\right]^{(2-q)/(1-q)}\eta\\
 =&\frac1{2-q}\left(\frac{f''}{f'}\right)^{2q-2} u^{(p-1)(2-q)}w^q\eta+\frac{1-q}{2-q}(6p)^{\frac{2-q}{1-q}}\left(\frac{f''}{f'}\right)^2w^2\eta\\\nonumber
 :=&\frac1{2-q}\left(\frac{f''}{f'}\right)^{2q-2} u^{(p-1)(2-q)}w^q\eta+C_1(p,q)\left(\frac{f''}{f'}\right)^2w^2\eta.
\end{align}
For $1<q<2p/(p+1)$, we have
\begin{align}\label{u^pw}\nonumber
&6pu^{p-1}w\eta\\
\leq&\frac{1}{q}\left[ \left(\frac{qM^2}{4N}\right)^{\frac{1}{q}}(f')^{\frac{2q-2}{q}}w\right]^q\eta+\frac{q-1}{q}\left[ 6p\left(\frac{qM^2}{4N}\right)^{-\frac{1}{q}}u^{p-1}(f')^{\frac{2-2q}{q}}\right]^{q/(q-1)}\eta\\\nonumber
:=&\frac{M^2}{4N}(f')^{2q-2}w^q\eta+C(N,p,q)M^{-\frac{2}{q-1}}u^{\frac{(p-1)q}{q-1}}(f')^{-2}\eta.
\end{align}
By using Young's inequality again, we obtain
\begin{equation}\label{Y-w_q+2}
-2M\left(q-\frac{N+1}{N}\right)|f'|^{q-2}f''w^{\frac{q+2}{2}}\eta\leq \frac{M^2}{4N}(f')^{2q-2}w^q\eta+C_2(N,q)\left(\frac{f''}{f'}\right)^2w^2\eta,
\end{equation}
and
\begin{align}\label{Malpha}\nonumber
  -M\alpha (q-1)|f'|^qu^{-1}w^{\frac{q+2}{2}}\eta&\leq M|\alpha||f'|^qu^{-1}w^{\frac{q+2}{2}}\eta\\
  &\leq \frac{M^2}{4N}(f')^{2q-2}w^q\eta+N\alpha^2(f')^2u^{-2}w^2\eta,
\end{align}
where we use the fact that $q<2p/(p+1)<2$. We select $-(3N)^{-1}<\alpha <0$ such that
\begin{equation}\label{-5N-1<alpha}
  N\alpha^2+\frac{\alpha(\alpha+1)}2<0.
\end{equation}
From \eqref{alpha<0}--\eqref{-5N-1<alpha}, for $q\leq 1$, we have
\begin{align}\label{alpha-beta}\nonumber
  &u^{-\alpha}\mathcal L_\alpha(u^\alpha w\eta)\\\nonumber
  \leq& \bigg[\frac{6p}{2-q}\left(\frac{f''}{f'}\right)^{2q-2} u^{(p-1)(2-q)}-\frac{M^2}{4N}(f')^{2q-2}\bigg]w^q\eta\\\nonumber
  &+2\left[\left(\frac{f''}{f'}\right)'+\left(C_1(p,q)+C_2(N,q)\right)\left(\frac{f''}{f'}\right)^2\right]w^2\eta+\frac{\alpha(\alpha+1)}2u^{-2}(f')^2w^2\eta\\
  &+w\left(|q|M|f'|^{q-1}w^{\frac{q-1}2}+2\left|\frac{f''}{f'}\right|w^{\frac12}+4\eta^{-1}|\nabla\eta|\right)|\nabla\eta|+\sqrt N w|D^2\eta|,
\end{align}
and for $1<q<2p/(p+1)$,
\begin{align}\label{alpha-beta-2}\nonumber
  &u^{-\alpha}\mathcal L_\alpha(u^\alpha w\eta)\\\nonumber
  \leq&-\frac{M^2}{4N}(f')^{2q-2}w^q\eta+2\bigg[\left(\frac{f''}{f'}\right)'+C_2(N,q)\left(\frac{f''}{f'}\right)^2\bigg]w^2\eta\\\nonumber
  &+C(N,p,q)M^{-\frac{2}{q-1}}u^{\frac{(p-1)q}{q-1}}(f')^{-2}\eta-\frac{1}{N}\frac{(-f)^{2p}}{(f')^2}\eta+\frac{\alpha(\alpha+1)}2u^{-2}(f')^2w^2\eta\\\nonumber
  &+ w\left(|q|M|f'|^{q-1}w^{\frac{q-1}2}+2\left|\frac{f''}{f'}\right|w^{\frac12}+4\eta^{-1}|\nabla\eta|\right)|\nabla\eta|+\sqrt N w|D^2\eta|\\
=&-\frac{M^2}{4N}(f')^{2q-2}w^q\eta+2\bigg[\left(\frac{f''}{f'}\right)'+C_2(N,q)\left(\frac{f''}{f'}\right)^2\bigg]w^2\eta\\\nonumber
&+M^{-\frac{2}{q-1}}(f')^{-2}u^{\frac{(p-1)q}{q-1}}\left(C(N,p,q)-\frac{1}{N }M^{\frac{2}{q-1}}u^{2p-\frac{(p-1)q}{q-1}}\right)\eta+\sqrt N w|D^2\eta|\\\nonumber
&+\frac{\alpha(\alpha+1)}2u^{-2}(f')^2w^2\eta+ w\left(|q|M|f'|^{q-1}w^{\frac{q-1}2}+2\left|\frac{f''}{f'}\right|w^{\frac12}+4\eta^{-1}|\nabla\eta|\right)|\nabla\eta|.
\end{align}

Letting  $\beta=1+1/C(N,p,q)$ in \eqref{f1}, where $C(N,p,q)=C_1(p,q)\chi_q+C_2(N,q)>0$, $\chi_q=1$ for $q\leq 1$ and $\chi_q=0$ for $q>1$, it follows that
\begin{equation}\label{f-beta-0}
  \left(\frac{f''}{f'}\right)'+C(N,p,q)\left(\frac{f''}{f'}\right)^2=0,
\end{equation}
\begin{equation}\label{m<f'<2m*}
   m\leq f'\leq 2\beta m,
\end{equation}
and
\begin{equation}\label{f''/f'<C}
  \left(\frac{f''}{f'}\right)^{2q-2}\leq C(N,p,q).
\end{equation}
Since $p\geq (N+3)/(N-1)$ and $q<2p/(p+1)<2$, we have $(p-1)(2-q)>0$.  By \eqref{m<f'<2m*} and \eqref{f''/f'<C}, we can choose a constant $c_{N,p,q}>0$ such that when
\begin{equation}\label{m-bound*}
 m=c_{N,p,q}M^{\frac{2}{2p-(p+1)q}},
\end{equation}
it holds that
\begin{equation*}
 \frac{6p}{2-q}\left(\frac{f''}{f'}\right)^{2q-2} u^{(p-1)(2-q)}-\frac{M^2}{4N}(f')^{2q-2}\leq 0 \quad \text{for}\ q\leq 1,
\end{equation*}
and
\begin{equation*}
 C(N,p,q)-\frac{1}{N}M^{\frac{2}{q-1}}u^{2p-\frac{(p-1)q}{q-1}}\leq 0\quad\text{for}\ 1< q<\frac{2p}{p+1}.
\end{equation*}
Using \eqref{alpha-beta}--\eqref{f-beta-0} yields
\begin{align*}
u^{-\alpha}\mathcal L_\alpha(u^\alpha w\eta)\leq &\frac{\alpha(\alpha+1)}2u^{-2}(f')^2w^2\eta+\sqrt Nw|D^2\eta|\\
&+w\left(M|q||f'|^{q-1}w^{\frac{q-1}2}+2\left|\frac{f''}{f'}\right|w^{\frac12}+4\eta^{-1}|\nabla\eta|\right)|\nabla\eta|.
\end{align*}

Set $z=u^\alpha w\eta$ and rewrite the above inequality as
\begin{align*}
  &u^{2-\alpha}(f')^{-2}w^{-2}\eta^{-1}\mathcal L_\alpha(z)\\
 \leq &\frac{\alpha(\alpha+1)}2+\sqrt N z^{-1} u^{\alpha+2}(f')^{-2}|D^2\eta|+M|q|z^{-\frac{3-q}2}u^{\frac{\alpha(3-q)+4}2}(f')^{q-3}\eta^{\frac{1-q}2}|\nabla\eta|\\
 &+2z^{-\frac12}u^{\frac{\alpha+4}2}\left|\frac{f''}{f'}\right|(f')^{-2}\eta^{-\frac12}|\nabla\eta|+4z^{-1}u^{\alpha+2}(f')^{-2}\eta^{-1}|\nabla\eta|^2.
\end{align*}
If $\max\{-(3N)^{-1},-4(3-q)^{-1}\}<\alpha <0$, \eqref{u<M1} together with the facts that $f'\geq C(N,M,p,q)$ and $|f''/f'|\leq C(N,M,p,q)$ leads to that there exists $C=C(N,M,p,q)>0$ such that
\begin{align*}
&u^{2-\alpha}(f')^{-2}w^{-2}\eta^{-1}\mathcal L_\alpha(z)\\
\leq & \frac{\alpha(\alpha+1)}2+C\left(z^{-1}|D^2\eta|+z^{-\frac{3-q}2}\eta^{\frac{1-q}2}|\nabla\eta|+z^{-\frac12}\eta^{-\frac12}|\nabla\eta|+z^{-1}\eta^{-1}|\nabla\eta|^2\right),
\end{align*}
which finishes the proof of Lemma \ref{lem:p>N+3,q<2p/(p+1)}.
\end{proof}

From Lemma \ref{lem:p>N+3,q<2p/(p+1)}, we can derive local gradient estimates for solutions.

\noindent\textbf{Proof of Theorem \ref{theorem:p>N+3,q<1}.}
Since $(1-q)/2>-1/2$, by taking $a=1/2$ in \eqref{dfi-eta} and using  properties of $\eta$, we deduce that when
\begin{equation*}
  z\geq A:=A(N,M,p,q,\alpha,R)=C_1\left(R^{-2}+R^{-\frac2{3-q}}\right)
\end{equation*}
for some $C_1=C(N,M,p,q,\alpha)>0$, it holds
\begin{equation}\label{<alpha/2}
C\left(z^{-1}|D^2\eta|+z^{-\frac{3-q}2}\eta^{\frac{1-q}2}|\nabla\eta|+z^{-\frac12}\eta^{-\frac12}|\nabla\eta|+z^{-1}\eta^{-1}|\nabla\eta|^2\right)<-\frac{\alpha(\alpha+1)}2.
\end{equation}
According to Lemma \ref{lem:p>N+3,q<2p/(p+1)}, it follows from \eqref{<alpha/2} that
\begin{equation*}
\mathcal L_\alpha(z)<0\quad \text{in}\ Q:=\left\{x\in B_{R'};\,z(x)> A\right\}.
\end{equation*}
Noting that $\eta\leq 1$ and $|\nabla u|=|f'||\nabla v|=|f'|w^{1/2}$, we obtain $|\nabla u|\geq \left(|f'|^2u^{-\alpha}A\right)^{1/2}\geq C$ in $Q$ for some $C=C(N,M,p,q,R)>0$. Utilizing the regularity theory on elliptic equation \eqref{eq1}, we know $u$ is in fact smooth on $Q$. Moreover, when $q<1$,
\begin{equation*}
   M|q||f'|^{q-1}w^{\frac{q-2}2}|\nabla v| =M|q||\nabla u|^{q-1}
\end{equation*}
in \eqref{defin-H} is bounded on $Q$. Hence $z\in C^2(Q)$ and $\mathcal H_\alpha\in L^{\infty}(Q)$. It follows from Lemma \ref{lem:maxi-pric} that
\begin{equation*}
  z\leq C_1\left(R^{-2}+R^{-\frac2{3-q}}\right)\quad\text{in}\ B_{R'}.
\end{equation*}
From $z=u^\alpha |\nabla v|^2\eta$ with $\eta=1$ in $B_{R/2}$, applying \eqref{m<f'<2m*} and \eqref{m-bound*} yields
\begin{align}\label{esti-z(x)}
 |\nabla u^{\frac{\alpha+2}2}|&=\frac{\alpha+2}2u^{\frac\alpha2}f'|\nabla v|\\\nonumber
 &\leq C(N,M,p,q,\alpha)m \left(R^{-1}+R^{-\frac1{3-q}}\right):=C\left(R^{-1}+R^{-\frac1{3-q}}\right)\quad\text{in}\ B_{R/2},
\end{align}
where $C=C(N,M,p,q,\alpha)>0$. Particularly, \eqref{esti-z(x)} holds at $x_0$. Following the arbitrariness of $x_0\in\Omega$, we conclude
\begin{equation*}
  |\nabla u^{\frac{\alpha+2}2}(x)|\leq C\left({\rm dist}^{-1}(x,\partial \Omega)+{\rm dist}^{-\frac1{3-q}}(x,\partial \Omega)\right),\quad x\in\Omega,
\end{equation*}
where $C=C(N,M,p,q,\alpha)>0$. Letting ${\rm dist}(x,\partial \Omega)\to \infty$ in the case $\Omega=\mathbb \R^N$, we obtain that the solution of \eqref{eq1} is a positive constant, which leads to a contradiction that $0<\int_{\mathbb \R^N} u^p\psi\leq0$ for some $0<\psi\in C^\infty_0(\mathbb \R^N)$. Therefore, we obtain the nonexistence of positive solutions and thus  Theorem \ref{theorem:p>N+3,q<1} follows.
\hfill$\Box$

\subsubsection{The supercritical range of $q$}

We next present the differential inequality of the gradient when $q$ is supercritical with respect to $p$.
\begin{lemma}\label{lem:q>2p/(p+1)}
Assume $u=-f(v)$ is a positive solution of \eqref{eq1} in $\Omega$ and $f$ defined in \eqref{f1}. Let $(p,q)\in G_2$. Then
\begin{align}\label{p>,q>}\nonumber
  \mathcal L(w\eta)\leq& -\frac{M^2}{2N}w^q\eta-\frac{1}{2N}u^{2p}\eta+C(N,p,q)M^{-\frac{4p}{(p+1)q-2p}}\eta\\
  &+C(N,M,q)\left(w\eta^{-1}|\nabla\eta|^2+w^{\frac{q+1}2}|\nabla\eta|+w|D^2\eta|\right)
\end{align}
holds in $\{x\in B_{R'};\,w(x)>0\}$.
\end{lemma}

\begin{proof}
Taking $\alpha=0$, $\gamma=1$, and $f(s)=-s$ in Lemma \ref{lem:L(wueta)}, it follows from  $f'=-1$ and $f''=0$ that
\begin{equation}\label{L-f=I}
  \mathcal L(w\eta)\leq -2|D^2v|^2\eta+2pu^{p-1}w\eta-2\langle\nabla w,\nabla\eta\rangle+Mqw^{\frac{q+1}2}|\nabla\eta|+\sqrt Nw|D^2\eta|.
\end{equation}
By Cauchy-Schwarz's and Young's inequalities, we know that
\begin{equation}\label{nablaw-nablaeta-2}
-2\langle\nabla w,\nabla\eta\rangle\leq 2|\nabla w||\nabla\eta|\leq 4|D^2v||\nabla v||\nabla\eta|\leq |D^2v|^2\eta+ 4w\eta^{-1}|\nabla \eta|^2.
\end{equation}
In view of \eqref{1/N} and \eqref{v-w}, we have
\begin{equation}\label{D2v-crit}
-|D^2v|^2\eta\leq -\frac1N\left(u^p+Mw^{\frac q2}\right)^2\eta\leq -\frac1Nu^{2p}\eta-\frac{M^2}{N}w^q\eta.
\end{equation}
Using Young's inequality and $2p/(p+1)<q$, we obtain
\begin{align}\label{u2p-M}\nonumber
   2pu^{p-1}w\eta&\leq \frac{1}{2N}u^{2p}\eta+C(N,p)w^{\frac{2p}{p+1}}\eta\\
   &\leq \frac{1}{2N}u^{2p}\eta+\frac{M^2}{2N}w^q\eta+C(N,p,q)M^{-\frac{4p}{(p+1)q-2p}}\eta.
\end{align}
Substituting \eqref{nablaw-nablaeta-2}--\eqref{u2p-M} into \eqref{L-f=I} leads to
\begin{align*}
  \mathcal L(w\eta)\leq& -\frac{M^2}{2N}w^q\eta-\frac{1}{2N}u^{2p}\eta+C(N,p,q)M^{-\frac{4p}{(p+1)q-2p}}\eta\\
  &+C(N,M,q)\left(w\eta^{-1}|\nabla\eta|^2+w^{\frac{q+1}2}|\nabla\eta|+w|D^2\eta|\right).
\end{align*}
The proof of Lemma \ref{lem:q>2p/(p+1)} is complete.
\end{proof}

\noindent\textbf{Proof of Theorem \ref{them:p>,q>}.}
Setting $z=w\eta$, inequality \eqref{p>,q>} can be rewritten as
\begin{align*}
  \mathcal L(z)\leq & z^q\eta^{1-q}\bigg[-\frac{M^2}{2N}+C(N,M,q)\Big(z^{-(q-1)}\eta^{q-3}|\nabla\eta|^2+z^{-\frac{q-1}2}\eta^{\frac {q-3}{2}}|\nabla\eta|\\
  &+z^{-(q-1)}\eta^{q-2}|D^2\eta|\Big)\bigg]-\frac{1}{2N}u^{2p}\eta+C(N,p,q)M^{-\frac{4p}{(p+1)q-2p}}\eta.
\end{align*}
Noting that $q>1$, we choose $a=\max\{0,2-q\}$ in \eqref{dfi-eta}, and thus
\begin{align}\label{u>M-q>}\nonumber
  \mathcal L(z)\leq & z^q\eta^{1-q}\bigg[-\frac{M^2}{2N}+C(N,M,q)\left(z^{-(q-1)}R^{-2}+z^{-\frac{q-1}2}R^{-1}\right)\bigg]\\
  &-\frac{1}{2N}u^{2p}\eta+C(N,p,q)M^{-\frac{4p}{(p+1)q-2p}}\eta.
\end{align}
Under the assumption \eqref{u>M2} for $u$, we obtain
\begin{equation*}
  \frac1{2N}u^{2p}>C(N,p,q)M^{-\frac{4p}{(p+1)q-2p}}.
\end{equation*}
Combining \eqref{u>M-q>}, we have
\begin{equation*}
  \mathcal L(z)<0\quad \text{in}\ \bigg\{x\in B_{R'};\ z(x)>C(N,M,q)R^{-\frac{2}{q-1}}\bigg\}.
\end{equation*}
Using similar arguments as the proof of Theorem \ref{theorem:p>N+3,q<1}, we derive the gradient estimates and nonexistence of solutions for \eqref{eq1} in $\mathbb \R^N$.
\hfill$\Box$

\subsection{The case $p<(N+3)/(N-1)$}

We first establish the differential inequality based on Lemma \ref{lem:e1}.
\begin{lemma}\label{lem:ieqp<}
Assume $u=-f(v)$ is a positive solution of \eqref{eq1} in $\Omega$ and $f$ defined by \eqref{f1}.  Then there exist $\gamma_0(N)>0$ and $C=C(N,p,q,\gamma)>0$ such that for any $\gamma\geq \gamma_0(N)$,
\begin{align}\label{q<,gamma>}\nonumber
  &w^{1-\gamma}\mathcal L(w^\gamma\eta)\\\nonumber
  \leq &2\gamma M\left(q-\frac{N+1}{N-1}\right)u^{q-1}w^{\frac{q+2}2}\eta+2\gamma\left[p-\frac{N+3}{N-1}+O\left(\gamma^{-\frac12}\right)\right]u^{p-1}w\eta\\
  &-w^2\eta-\frac{\gamma M^2}{2(N-1)}u^{2q-2}w^q\eta+2w^{\frac32}|\nabla\eta|+C\left(\eta^{-3}|\nabla\eta|^4+\eta^{-1}|D^2\eta|^2\right)
\end{align}
holds in $\{x\in B_{R'};\,w(x)>0\}$.
\end{lemma}

\begin{proof}
We choose
\begin{equation*}
  f(s)=-e^s,
\end{equation*}
and then $f$ maps $(-\infty,\infty)$ to $(-\infty,0)$. Thus, $f',f''<0$ and
\begin{equation}\label{f=es}
  \frac{f''u}{(f')^2}=-1,\quad \frac{f''}{f'}=1,\quad  \left(\frac{f''}{f'}\right)'=0.
\end{equation}
Owing to Lemma \ref{lem:e1}, we follow from \eqref{f=es} that
\begin{align}\label{u^(2p-2)young}\nonumber
   &w^{1-\gamma}\mathcal L(w^{\gamma}\eta)\\\nonumber
  \leq &\gamma \left(-2(\gamma-1)+\frac{2(1-N+3\gamma^{1/2})}{N-1}\right)v^2_{11}\eta-2\gamma M\left(q-\frac{N+1}{N-1}\right)|f'|^{q-2}f''w^{\frac{q+2}2}\eta\\\nonumber
  &+2\gamma\left(p-\frac{N+1}{N-1}\right)u^{p-1}w\eta+\left(-\frac{2\gamma(1-\gamma^{-1/2})}{N-1}+3\right)w^2\eta-\frac{2\gamma(1-\gamma^{-1/2})}{N-1}u^{2p-2}\eta\\
  &-\frac{\gamma (1-2\gamma^{-1/2})}{N-1}M^2|f'|^{2q-2}w^q\eta+2w^{\frac32}|\nabla\eta|+C\left(\eta^{-3}|\nabla\eta|^4+\eta^{-1}|D^2\eta|^2\right)\\\nonumber
  =&\gamma \left(-2(\gamma-1)+\frac{2(1-N+3\gamma^{1/2})}{N-1}\right)v^2_{11}\eta-2\gamma M\left(q-\frac{N+1}{N-1}\right)|f'|^{q-2}f''w^{\frac{q+2}2}\eta-w^2\eta\\\nonumber
  &+2\gamma\left(p-\frac{N+1}{N-1}\right)u^{p-1}w\eta-\frac{2\gamma\left(1-\gamma^{-1/2}-2\gamma^{-1}(N-1)\right)}{N-1}w^2\eta-\frac{2\gamma(1-\gamma^{-1/2})}{N-1}u^{2p-2}\eta\\\nonumber
  &-\frac{\gamma (1-2\gamma^{-1/2})}{N-1}M^2|f'|^{2q-2}w^q\eta+2w^{\frac32}|\nabla\eta|+C\left(\eta^{-3}|\nabla\eta|^4+\eta^{-1}|D^2\eta|^2\right),
\end{align}
where $C=C(N,p,q,\gamma)>0$. By Young's inequality, we see that for $\gamma>\gamma_1(N):=16(N-1)^2$,
\begin{align*}
&-\frac{2\gamma}{N-1}\left[\left(1-\gamma^{-\frac12}-2\gamma^{-1}(N-1)\right)w^2+(1-\gamma^{-\frac12})u^{2p-2}\right]\eta\\
\leq&-\frac{4\gamma}{N-1}\left(1-\gamma^{-\frac12}-2\gamma^{-1}(N-1)\right)^{\frac12}(1-\gamma^{-\frac12})^{\frac12}u^{p-1}w\eta.
\end{align*}
Together with \eqref{u^(2p-2)young}, it leads to
\begin{align*}
  &w^{1-\gamma}\mathcal L(w^\gamma\eta)\\
  \leq &\gamma \left(-2(\gamma-1)+\frac{2(1-N+3\gamma^{1/2})}{N-1}\right)v^2_{11}\eta-2\gamma M\left(q-\frac{N+1}{N-1}\right)|f'|^{q-2}f''w^{\frac{q+2}2}\eta-w^2\eta\\\nonumber
  &+2\gamma\left[p-\frac{N+1}{N-1}-\frac{2}{N-1}\left(1-\gamma^{-\frac12}-2\gamma^{-1}(N-1)\right)^{\frac12}\left(1-\gamma^{-\frac12}\right)^{\frac12}\right]u^{p-1}w\eta\\\nonumber
  &-\frac{\gamma M^2(1-2\gamma^{-1/2})}{N-1}|f'|^{2q-2}w^q\eta+2w^{\frac32}|\nabla\eta|+C\left(\eta^{-3}|\nabla\eta|^4+\eta^{-1}|D^2\eta|^2\right),
\end{align*}
where $C=C(N,p,q,\gamma)>0$. Note that
\begin{equation*}
  \left(1-\gamma^{-\frac12}-2\gamma^{-1}(N-1)\right)^{\frac12}\left(1-\gamma^{-\frac12}\right)^{\frac12}=1-\gamma^{-\frac12}+o\left(\gamma^{-\frac12}\right)\quad\text{as}\ \gamma\to\infty.
\end{equation*}
Hence, there exists $\gamma_0(N)\geq \gamma_1(N)$ large enough, such that for any $\gamma\geq \gamma_0(N)$,
\begin{equation*}
  p-\frac{N+1}{N-1}-\frac{2}{N-1}\left(1-\gamma^{-\frac12}-2\gamma^{-1}(N-1)\right)^{\frac12}\left(1-\gamma^{-\frac12}\right)^{\frac12}=p-\frac{N+3}{N-1}+O\left(\gamma^{-\frac12}\right).
\end{equation*}
In addition, the coefficient of $v_{11}^2\eta$ is negative for $\gamma\geq \gamma_0(N)$. This completes the proof of Lemma \ref{lem:ieqp<}.
\end{proof}

\subsubsection{Proof of Theorem \ref{them:p,q<n+3,n+1}} Since $(p,q)\in G_3$,  there exists $\gamma_1(N,p)>0$ such that
\begin{equation*}
p-\frac{N+3}{N-1}+O\left(\gamma^{-\frac12}\right)<0
\end{equation*}
for $\gamma\geq \gamma_1$. From now on, we fix
\begin{equation*}
  \gamma=\gamma_1(N,p),
\end{equation*}
and then coefficients of $u^{q-1}w^{(q+2)/2}\eta$ and $u^{p-1}w\eta$ in \eqref{q<,gamma>} are negative. Furthermore, it follows from Lemma \ref{lem:ieqp<} that
\begin{align}\label{-w^gamma+1}
  \mathcal L(w^{\gamma}\eta)\leq -w^{\gamma+1}\eta+2w^{\gamma+\frac12}|\nabla\eta|+Cw^{\gamma-1}\left(\eta^{-3}|\nabla\eta|^4+\eta^{-1}|D^2\eta|^2\right),
\end{align}
where $C=C(N,p,q)>0$. Now, we eliminate terms containing $|\nabla\eta|$ and $|D^2\eta|$. From Young's inequality with exponents $2(\gamma+1)/(2\gamma+1)$ and $2(\gamma+1)$, we find
\begin{equation*}
  2w^{\gamma+\frac12}|\nabla\eta|\leq \frac14w^{\gamma+1}\eta+ C(N,p)\left(\eta^{-\frac{2\gamma+1}{2(\gamma+1)}}|\nabla\eta|\right)^{2(\gamma+1)}.
\end{equation*}
By the same trick with exponents $(\gamma+1)/(\gamma-1)$ and $(\gamma+1)/2$, we have
\begin{equation*}
  Cw^{\gamma-1}\left(\eta^{-3}|\nabla\eta|^4+\eta^{-1}|D^2\eta|^2\right)\leq \frac14w^{\gamma+1}\eta+C\left(\eta^{-\frac{2\gamma+1}{2(\gamma+1)}}|\nabla\eta|\right)^{2(\gamma+1)}+C\left(\eta^{-\frac{\gamma}{\gamma+1}}|D^2\eta|\right)^{\gamma+1},
\end{equation*}
where $C=C(N,p,q)>0$. Setting $z=w^\gamma\eta$ and using properties of $\eta$ in \eqref{dfi-eta} with $a=(2\gamma+1)/(2\gamma+2)<1$, we have
\begin{align*}
  \mathcal L(z)&\leq-\frac12w^{\gamma+1}\eta+ CR^{-2(\gamma+1)}\\
  &=-z^{\frac{\gamma+1}{\gamma}}\eta^{1-\frac{\gamma+1}{\gamma}}\left(\frac12-Cz^{-\frac{\gamma+1}{\gamma}}\eta^{\frac{\gamma+1}{\gamma}-1}R^{-2(\gamma+1)}\right)\\
  &\leq -z^{\frac{\gamma+1}{\gamma}}\eta^{1-\frac{\gamma+1}{\gamma}}\left(\frac12-Cz^{-\frac{\gamma+1}{\gamma}}R^{-2(\gamma+1)}\right).
\end{align*}
There exists $C_1=C(N,p,q)>0$ such that when $z\geq C_1R^{-2\gamma}$,
\begin{equation*}
 Cz^{-\frac{\gamma+1}{\gamma}}R^{-2(\gamma+1)}<\frac12.
\end{equation*}
Therefore, it follows that
\begin{equation*}
  \mathcal L(z)<0\quad\text{in}\ \{x\in B_{R'};\,z(x)>C_1R^{-2\gamma}\}.
\end{equation*}
Essentially as in the proof of Theorem \ref{theorem:p>N+3,q<1}, we have $z\leq C_1 R^{-2\gamma}$ in $B_{R'}$. From $z=w^{\gamma}\eta$ and the definition of $\eta$, we have
\begin{equation*}
  \frac{|\nabla u|}{u}=\frac{|\nabla u|}{|f'|}=|\nabla v|=w^{\frac12}\leq CR^{-1}\quad\text{in}\ B_{R/2},
\end{equation*}
where $C=C(N,p,q)>0$. In particular, the estimates still holds for $x_0$. From the arbitrariness of $x_0\in\Omega$, we obtain the estimates for $x\in\Omega$. Assume that $u$ is a positive solution to \eqref{eq1} in $\mathbb \R^N$, we can deduce $0<\int_{\mathbb \R^N} u^p\psi\leq0$ for some $0<\psi\in C^\infty_0(\mathbb \R^N)$, which is a contradiction. Hence, we complete the proof.
\hfill$\Box$

\subsubsection{Proof of Theorem \ref{them:p<=0,q>}}

Taking $\alpha=0$, $\gamma=1$ and $f(s)=-s$ in Lemma \ref{lem:L(wueta)}, we follow from
\begin{equation*}
  f'=-1,\quad f''=0
\end{equation*}
that
\begin{align*}
  \mathcal L(w\eta)\leq -2|D^2v|^2\eta +2p u^{p-1}w\eta-2\langle\nabla w,\nabla\eta\rangle +qMw^{\frac{q+1}2}|\nabla \eta|+\sqrt Nw|D^2\eta|.
\end{align*}
According to \eqref{exp-D2v}, \eqref{nablaw.nablaeta} and $p\leq 0$,  we have
\begin{align*}
  \mathcal L(w\eta)\leq& -\frac1N\left(u^p+Mw^{\frac{q}2}\right)^2\eta+4w\eta^{-1}|\nabla\eta|^2+Mqw^{\frac{q+1}2}|\nabla\eta|+\sqrt Nw|D^2\eta|\\
  \leq &-\frac{M^2}Nw^q \eta+4w\eta^{-1}|\nabla\eta|^2+qMw^{\frac{q+1}2}|\nabla\eta|+\sqrt Nw|D^2\eta|.
\end{align*}
Setting $z=w\eta$, we write the above inequality as
\begin{align*}
  &\mathcal L(z)\\
  \leq& z^q\eta^{1-q}\bigg[-\frac{M^2}{N}+\Big(4z^{-(q-1)}\eta^{q-3}|\nabla\eta|^2+qMz^{-\frac{q-1}2}\eta^{\frac {q-3}{2}}|\nabla\eta|
  +\sqrt Nz^{-(q-1)}\eta^{q-2}|D^2\eta|\Big)\bigg].
\end{align*}
We choose $a=\max \{0,2-q\}$ in \eqref{dfi-eta}, and thus
\begin{equation}\label{L(z):q<0}
  \mathcal L(z)\leq z^q\eta^{1-q}\bigg[-\frac{M^2}N+C(N,M,q)\left(z^{-(q-1)}R^{-2}+z^{-\frac{q-1}2}R^{-1}\right)\bigg],
\end{equation}
which implies that there exists $C=C(N,M,q)>0$ such that
\begin{equation*}
  \mathcal L(z)<0\quad \text{in}\ \Big\{x\in B_{R'};\,z(x)>CR^{-\frac2{q-1}}\Big\}.
\end{equation*}
Using similar arguments as the proof of Theorem \ref{theorem:p>N+3,q<1}, and combining with $f'=-1$, we derive
\begin{equation*}
  |\nabla u(x)|=w^{\frac12}(x)\leq C(N,M,q){\rm dist}^{-\frac1{q-1}}(x,\partial\Omega),\quad x\in\Omega,
\end{equation*}
and hence the solution of \eqref{eq1} in $\mathbb R^N$ is a positive constant, which contradicts that $0<\int_{\mathbb \R^N} u^p\psi\leq0$ for some $0<\psi\in C^\infty_0(\mathbb \R^N)$. Therefore, we derive the nonexistence of positive solutions and complete the proof of Theorem \ref{them:p<=0,q>}.
\hfill$\Box$

\subsubsection{Proof of Theorem \ref{them:0<p<,q>}}

Observing  $q\geq (N+1)/(N-1)$ for $(p,q)\in G_5$, we have
\begin{equation*}
0\leq 2\gamma M\left(q-\frac{N+1}{N-1}\right)u^{q-1}w^{\frac{q+2}{2}}\eta\leq 2\gamma M qu^{q-1}w^{\frac{q+2}2}\eta.
\end{equation*}
Clearly, the estimates of $2\gamma M qu^{q-1}w^{(q+2)/2}\eta$ varies from the range of $q$. For $(N+1)/(N-1)\leq q<2$, it has $1<(q+2)/2<2$. Using Young's inequality with conjugate exponents $2/(2-q)$ and $2/q$, we get
\begin{align}\label{1<q<2}\nonumber
&2\gamma qMu^{q-1}w^{\frac{q+2}{2}}\eta\\\nonumber
\leq&\frac{(2-q)\gamma}{2}\Bigg[\gamma^{-\frac{(2-q)}4} \left(u^{p-1}w\right)^{\frac{2-q}2} \Bigg]^{2/(2-q)}\eta+\frac{q\gamma}{2}\Bigg[2\gamma^{\frac{2-q}4}Mq\left(u^{p-1}\right)^{\frac{q-2}2}u^{q-1}w^{q}\Bigg]^{2/q}\eta\\
:=&\frac{(2-q)}{2}\gamma^{\frac12} u^{p-1}w\eta+C(q,\gamma)M^{\frac2q}u^{\frac{(p+1)q-2p}{q}}w^2\eta.
\end{align}
For $q=2$, we have
\begin{equation}\label{q=2}
2\gamma Mqu^{q-1}w^{\frac{q+2}{2}}\eta=4\gamma Muw^2\eta.
\end{equation}
For $q>2$, it is easy to see that $1<(q+2)/2<q$. By Young's inequality with conjugate exponents $2(q-1)/(q-2)$ and $2(q-1)/q$, we find that
\begin{align}\label{q>2}\nonumber
 &2\gamma Mqu^{q-1}w^{\frac{q+2}{2}}\eta\\\nonumber
\leq& \frac{(q-2)\gamma}{2(q-1)} u\left[\gamma ^{-\frac{q-2}{4(q-1)}}\left(u^{p-2}w\right)^{\frac{q-2}{2(q-1)}}\right]^{2(q-1)/(q-2)}\eta\\\nonumber
&+\frac{q\gamma}{2(q-1)} u\left[2\gamma ^{\frac{q-2}{4(q-1)}}Mqu^{q-2}\left(u^{p-2}\right)^{\frac{-(q-2)}{2(q-1)}}w^{\frac{q^2}{2(q-1)}}\right]^{2(q-1)/q}\eta\\
:=&\frac{q-2}{2(q-1)}\gamma^{\frac12}u^{p-1}w\eta+C(q,\gamma)M^{\frac{2(q-1)}{q}}u^{2q-3+\frac{2p-pq}q} w^{q}\eta.
\end{align}

Now we return to estimate $\mathcal L(w^\gamma\eta)$. In the case of $(N+1)/(N-1)\leq q\leq 2$, substituting \eqref{1<q<2} or \eqref{q=2} into \eqref{q<,gamma>}, we obtain
\begin{align}\label{q<2-ineq}\nonumber
  \mathcal L(w^\gamma\eta)\leq & 2\gamma\left[p-\frac{N+3}{N-1}+O\left(\gamma^{-\frac12}\right)\right]u^{p-1}w^\gamma\eta-\left(\frac12-CM^{\frac2q}u^{\frac{(p+1)q-2p}{q}}\right)w^{\gamma+1}\eta\\
  &-\frac12 w^{\gamma+1}\eta+2w^{\gamma+\frac12}|\nabla\eta|+Cw^{\gamma-1}\left(\eta^{-3}|\nabla\eta|^4+\eta^{-1}|D^2\eta|^2\right),
\end{align}
where $C=C(N,p,q,\gamma)>0$. For $q>2$, it follows from \eqref{q>2} and \eqref{q<,gamma>} that
\begin{align}\label{q>2-ineq}\nonumber
\mathcal L(w^{\gamma}\eta)\leq&2\gamma\left[p-\frac{N+3}{N-1}+O\left(\gamma^{-\frac12}\right)\right]u^{p-1}w^\gamma\eta\\
&-\gamma\left(\frac{1}{2(N-1)}-CM^{-\frac2q}u^{\frac{2p-(p+1)q}{q}}\right)M^2u^{2q-2}w^{\gamma-1+q}\eta\\\nonumber
 &-\frac1{2}w^{\gamma+1}\eta+2w^{\gamma+\frac12}|\nabla\eta|+Cw^{\gamma-1}\left(\eta^{-3}|\nabla\eta|^4+\eta^{-1}|D^2\eta|^2\right),
\end{align}
where $C=C(N,p,q,\gamma)>0$. Since $p<(N+3)/(N-1)$, there exists $\gamma_1(N,p)>0$ such that for $\gamma\geq \gamma_1(N,p)$,
\begin{equation*}
  p-\frac{N+3}{N-1}+O\left(\gamma^{-\frac12}\right)<0.
\end{equation*}
From now on, we set
\begin{equation*}
  \gamma=\gamma_1(N,p).
\end{equation*}
Under the assumption of $u$ in Theorem \ref{them:0<p<,q>}, and the fact that $2p/(p+1)<(N+1)/(N-1)$, we can obtain
\begin{equation*}
  \frac12-C(N,p,q)M^{\frac2q}u^{\frac{(p+1)q-2p}{q}}\geq 0\quad\text{for}\ \frac{N+1}{N-1}\leq q\leq 2,
\end{equation*}
and
\begin{equation*}
\frac{1}{2(N-1)}-C(N,p,q)M^{-\frac2q}u^{\frac{2p-(p+1)q}{q}}\geq 0\quad \text{for}\ q>2.
\end{equation*}
It follows from \eqref{q<2-ineq} and \eqref{q>2-ineq} that
\begin{equation*}
  \mathcal L(w^\gamma\eta)\leq  -\frac12w^{\gamma+1}\eta+2w^{\gamma+\frac12}|\nabla\eta|+Cw^{\gamma-1}\left(\eta^{-3}|\nabla\eta|^4+\eta^{-1}|D^2\eta|^2\right)
\end{equation*}
with $C=C(N,p,q)>0$, which is similar to \eqref{-w^gamma+1}. Hence, the rest of the proof follows from previous arguments and  we thus complete the proof.
\hfill$\Box$

\section{Universal estimates via Liouville-type theorem}\label{sec:uni-est}
In this section, we shall establish estimates for solutions. The following lemma provides gradient estimates for bounded positive solutions when $q$ is critical with respect to $p$.

\begin{proposition}\label{pro:q=criti}
Let $(p,q)\in G_6$. Assume $u$ is a positive solution of \eqref{eq1} in $\Omega$ which satisfies $u\leq m$ for some constant $m>0$. Then there exists a constant $C=C(N,p)>0$ such that for any
\begin{equation*}
  M\geq M_0:=2^{-\frac{p-1}{p+1}}(6N)^{\frac{p}{p+1}}(N+2)^{\frac{p-1}{p+1}}(p-1)^{\frac{p-1}{p+1}}p^{-\frac{p}{p+1}}(p+1),
\end{equation*}
it holds that
\begin{equation*}
  |\nabla u|\leq C(N,p)mR^{-1}\quad\text{in}\ B_{R/2}.
\end{equation*}
\end{proposition}

\begin{proof}
We select $\alpha=0$ and $\gamma=1$ in \eqref{ineq-alpha-gamma}, then
\begin{align*}
\mathcal L(w\eta)\leq& -2|D^2v|^2\eta+\left(2p+\frac{2f'' u}{(f')^2}\right)u^{p-1}w\eta+2\left(\frac{f''}{f'}\right)'w^2\eta-2(q-1)M|f'|^{q-2}f''w^{\frac{q+2}2}\eta\\
&-2\langle\nabla w,\nabla\eta\rangle+ qM|f'|^{q-1}w^{\frac{q+1}2}|\nabla\eta|+2\left|\frac{f''}{f'}\right|w^{\frac32}|\nabla\eta|+\sqrt Nw|D^2\eta|.
\end{align*}
Through similar arguments as  \eqref{nablaw-nablaeta-2} and \eqref{exp-D2v}, we have
\begin{align*}
-2|D^2v|^2\eta-2\langle\nabla w,\nabla\eta\rangle\leq-\frac1N|\Delta v|^2\eta+4w\eta^{-1}|\nabla\eta|^2,
\end{align*}
and
\begin{align*}
&-\frac1N|\Delta v|^2\eta\\
=&-\frac1N\left(\frac{(-f)^p}{f'}+M\frac{|f'|^q}{f'}w^{\frac{q}2}\right)^2\eta-\frac1N\bigg[\left(\frac{f''}{f'}\right)^2w^2-2\frac{(-f)^pf''}{(f')^2}w-2M|f'|^{q-2}f''w^{\frac{q+2}{2}}\bigg]\eta.
\end{align*}
It follows from Young's inequality that
\begin{equation*}
\frac{2M}N|f'|^{q-2}f''w^{\frac{q+2}2}\eta\leq \frac{M^2}{4N}|f'|^{2q-2}w^q\eta+\frac4N\left(\frac{f''}{f'}\right)^2w^2\eta.
\end{equation*}
From inequalities above, we derive
\begin{align}\label{Lw-4}\nonumber
\mathcal L(w\eta)\leq&2\left(p+\frac{N+1}N\frac{f''u}{(f')^2}\right)u^{p-1}w\eta+2\left[\left(\frac{f''}{f'}\right)'+\frac{2}{N}\left(\frac{f''}{f'}\right)^2\right]w^2\eta-\frac{1}{N}\left(\frac{f''}{f'}\right)^2w^2\eta\\
&-2M(q-1)(f')^{q-2}f''w^{\frac{q+2}2}\eta-\frac{1}{N}\left(\frac{(-f)^{p}}{f'}+\frac{M}2\frac{|f'|^q}{f'}w^{\frac q2}\right)^2\eta\\\nonumber
&+4w\eta^{-1}|\nabla \eta|^2+qM|f'|^{q-1}w^{\frac{q+1}2}|\nabla \eta|+2\left|\frac{f''}{f'}\right|w^{\frac32}|\nabla\eta|+\sqrt N w|D^2\eta|.
\end{align}

For given $m>0$, choosing the function $f$ defined in \eqref{f1} with
\begin{equation*}\label{beta=N+2/2}
\beta=\frac{N+2}2,
\end{equation*}
we know $f',f''>0$ and
\begin{equation*}
\left(\frac{f''}{f'}\right)'+\frac{2}{N}\left(\frac{f''}{f'}\right)^2=0.
\end{equation*}
Furthermore, $f$ maps $[0,2^{1/\beta}-1)$ into $[-m,0)$ and
\begin{equation*}
\frac{f''u}{(f')^2}\leq \frac{\beta-1}{\beta}<1,\quad (f')^{q-2}f''\geq C_1m^{q-1},\quad C_2\leq\left|\frac{f''}{f'}\right|\leq C_3,
\end{equation*}
and
\begin{equation*}
  m\leq f'\leq 2\beta m,
\end{equation*}
where $C_1=C(N,q)>0$ and $C_i=C(N)>0$ for $i\in\{2,3\}$.
Then it follows from \eqref{Lw-4} that
\begin{align}\label{Lw-M0}\nonumber
  \mathcal L(w\eta)\leq& 6pu^{p-1}w\eta-C_2w^2\eta-C_1Mm^{q-1}w^{\frac{q+2}2}\eta-\frac{1}{4N}\left(\frac{u^{p}}{\beta m}+Mm^{q-1}w^{\frac q2}\right)^2\eta\\
  &+4w\eta^{-1}|\nabla \eta|^2+C_4Mm^{q-1}w^{\frac{q+1}2}|\nabla \eta|+C_3w^{\frac32}|\nabla\eta|+\sqrt N w|D^2\eta|,
\end{align}
where $C_4=C(N,q)>0$.

We now determine the sign of $\mathcal L(w\eta)$ using the strategy from \cite{Veron-MathAnn-20}. Since $q=2p/(p+1)$, we have
\begin{align*}
  &6pu^{p-1}w-\frac{1}{4N}\left(\frac{u^{p}}{\beta m}+Mm^{q-1}w^{\frac{q}2}\right)^2\\
:=&\Phi(w)\bigg[\sqrt6p^{\frac12}u^{\frac{p-1}2}w^{\frac12}+(4N)^{-\frac12}\left(\frac{u^{p}}{\beta m}+Mm^{q-1}w^{\frac{q}2}\right)\bigg],
  \end{align*}
where
\begin{equation*}
  \Phi(w)=\sqrt6p^{\frac12}u^{\frac{p-1}2}w^{\frac12}-(4N)^{-\frac12}\left(\frac{u^{p}}{\beta m}+Mm^{\frac{p-1}{p+1}}w^{\frac{p}{p+1}}\right).
\end{equation*}
By a direct calculation, we know that $\Phi$ achieves its maximum at
\begin{equation*}
  w_0=\left(\sqrt{6N}(p+1)p^{-\frac12}M^{-1}m^{-\frac{p-1}{p+1}}u^{\frac{p-1}2}\right)^{2(p+1)/(p-1)},
\end{equation*}
and hence
\begin{align*}
  \Phi(w)\leq &\Phi(w_0)\\
 =&\left(6^{\frac{p}{p-1}}N^{\frac{p+1}{2(p-1)}}(p-1)p^{-\frac{p}{p-1}}(p+1)^{\frac{p+1}{p-1}}M^{-\frac{p+1}{p-1}}-N^{-\frac12}\beta^{-1}\right)\frac{u^{p}}{2m}.
\end{align*}
If we select
\begin{equation*}
  M\geq M_0:=2^{-\frac{p-1}{p+1}}(6N)^{\frac{p}{p+1}}(N+2)^{\frac{p-1}{p+1}}(p-1)^{\frac{p-1}{p+1}}p^{-\frac{p}{p+1}}(p+1),
\end{equation*}
then $\Phi (w)\leq 0$. It follows from \eqref{Lw-M0} that
\begin{align*}
\mathcal L(w\eta)&\leq -C_2w^2\eta-C_1Mm^{\frac{p-1}{p+1}}w^{\frac{2p+1}{p+1}}\eta+4w\eta^{-1}|\nabla \eta|^2\\
 &+C_4Mm^{\frac{p-1}{p+1}}w^{\frac{3p+1}{2(p+1)}}|\nabla \eta|+C_3w^{\frac32}|\nabla\eta|+\sqrt N w|D^2\eta|.
\end{align*}
Letting $z=w\eta$, the above inequality implies
\begin{align*}
  \mathcal L(z)\leq& -z^2\eta^{-1}\left(C_2-C_3z^{-\frac12}\eta^{-\frac12}|\nabla\eta|-4z^{-1}\eta^{-1}|\nabla \eta|^2-\sqrt N z^{-1}|D^2\eta|\right)\\
  &-Mm^{\frac{p-1}{p+1}}z^{\frac{q+2}2}\eta^{-\frac{q}2}\left(C_1-C_4z^{-\frac12}\eta^{-\frac12}|\nabla\eta|\right).
\end{align*}
Taking $a=1/2$ in \eqref{dfi-eta}, there exists $C=C(N,p)>0$ such that if $z\geq C(N,p)R^{-2}$, it holds
\begin{equation*}
4z^{-1}\eta^{-1}|\nabla \eta|^2+(C_3+C_4)z^{-\frac12}\eta^{-\frac12}|\nabla\eta|+\sqrt N z^{-1}|D^2\eta|< C_5,
\end{equation*}
where $C_5:=\min\{C_1,C_2\}$. Hence,
\begin{equation*}
  \mathcal L(z)<0\quad\text{in}\ \left\{x\in B_{R'};\,z(x)\geq C(N,p)R^{-2}\right\}.
\end{equation*}
It follows from Lemma \ref{lem:maxi-pric} that
\begin{equation*}
  z\leq C(N,p)R^{-2}\quad\text{in}\ B_{R'}.
\end{equation*}
Finally, using $z=|\nabla v|^2\eta$, $|\nabla u|=f'|\nabla v|$, and the definition of $\eta$, we have
\begin{equation*}
|\nabla u|\leq C(N,p)m R^{-1}\quad\text{in}\ B_{R/2}.
\end{equation*}
Thus, we finish the proof.
\end{proof}

From the local gradient estimates for positive bounded solutions, letting $R\to\infty$, we can obtain the following result.

\begin{corollary}\label{corol-u-bound}
Let $(p,q)\in G_6$. Then for any $M\geq M_0$, where $M_0$ is defined in \eqref{defi-M0}, \eqref{eq1} possesses no positive bounded solution in $\mathbb R^N$.
\end{corollary}

We recall the useful doubling lemma due to \cite[Lemma 5.1]{Polacik-Quittner-Souplet}, which allows us to structure the rescaling procedure and derive a contradiction. The doubling property is an extension of an idea by Hu \cite{Hu-96-doubling}.
\begin{lemma}\label{lem:doub}
Let $(X, d)$ be a complete metric space with metric $d$, and let $\emptyset \neq D \subset \Sigma \subset X$ with $\Sigma$ closed. Set $\Gamma=\Sigma \backslash D$. Let $M: D \rightarrow(0, \infty)$ be bounded on compact subsets of $D$, and fix a real number $k>0$. If $y \in D$ is such that
\begin{equation*}
  M(y){\rm{dist}}(y,\Gamma)>2k,
\end{equation*}
then there exists $x\in D$ such that
\begin{equation*}
  M(x){\rm{dist}}(x,\Gamma)>2k,\quad M(x)\geq M(y),
\end{equation*}
and
\begin{equation*}
  M(z)\leq 2M(x)\ \ \text{\rm for all}\ z\in D\cap\overline{B}_{X}\left(x,kM^{-1}(x)\right).
\end{equation*}
\end{lemma}

\noindent\textbf{Proofs of Theorems \ref{them:uni-esti} and \ref{them:f-g}.} We set
\begin{equation*}
  M(u)=u^{\frac{p-1}2}+|\nabla u|^{\frac{p-1}{p+1}}.
\end{equation*}
The estimates \eqref{uni-esti} and \eqref{uni-f+g} are equivalent to proving
\begin{equation*}
  M(u(x))\leq C\left(\sigma+{\rm dist}^{-1}(x,\partial\Omega)\right),\quad x\in\Omega
\end{equation*}
with $C>0$ and
\begin{equation*}
  \sigma:=\left\{
\begin{aligned}
  &0, &&\text{under the assumptions of Theorem \ref{them:uni-esti}},\\
  &1, &&\text{under the assumptions of Theorem \ref{them:f-g}}.
\end{aligned}
\right.
\end{equation*}
Assume for contradiction that there exist sequences of domains $\Omega_k$, solutions $u_k$ of \eqref{eq1} or \eqref{eq-f+g} in $\Omega_k$, respectively, and points $y_k\in\Omega_k$ such that
\begin{equation*}
M_k(y_k):=M(u_k(y_k))>2k\left(\sigma+{\rm dist}^{-1}\left(y_k,\partial \Omega_k\right)\right)\geq 2k{\rm dist}^{-1}(y_k,\partial \Omega_k).
\end{equation*}

Since $p>1$, we know the function $M_k$ is continuous and hence locally bounded on $\Omega_k$ for each $k>0$. By Lemma \ref{lem:doub} with $X=\mathbb R^N$, there exist $x_k\in\Omega_k$ such that
\begin{equation}\label{M-xk}
  M_k(x_k)>2k{\rm dist}^{-1}(x_k,\partial \Omega_k),\quad M_k(x_k)\geq M_k(y_k)>2k\sigma,
\end{equation}
and
\begin{equation}\label{M<2Mk}
  M_k(x)\leq 2M_k(x_k)\quad \text{in}\ \widehat \Omega_k:=\left\{x\in\mathbb R^{N};\ {\rm dist}\left(x,x_k\right)\leq kM_k^{-1}(x_k)\right\}.
\end{equation}
Owing to the first inequality of \eqref{M-xk}, we get
\begin{equation*}
  {\rm dist}\left(x,x_k\right)<\frac12{\rm dist}\left(x_k,\partial\Omega_k\right),\quad x\in\widehat \Omega_k,
\end{equation*}
which implies $\widehat \Omega_k\subset\Omega_k$. Let
\begin{equation*}
\lambda_k=M^{-1}_k(x_k).
\end{equation*}
Now, we rescale $u_k$ by setting
\begin{equation}\label{scal}
  v_k(y)=\lambda _k^{2/(p-1)}u_k\left(x_k+\lambda_ky\right),\quad |y|\leq k.
\end{equation}
Then the function $v_k$ solves
\begin{equation*}
  -\Delta v_k=v^p_k+M|\nabla v_k|^{2p/(p+1)}\quad\text{in}\ |y|\leq k,
\end{equation*}
or
\begin{equation*}
-\Delta v_k=\lambda_k^{2p/(p-1)}f\left(\lambda_k^{-2/(p-1)}v_k\right)+\lambda_k^{2p/(p-1)}g\left(\lambda_k^{-(p+1)/(p-1)}\nabla v_k\right)\quad \text{in}\ |y|\leq k
\end{equation*}
under assumptions of Theorems \ref{them:uni-esti} or \ref{them:f-g}, respectively. Moreover, from \eqref{M<2Mk} and \eqref{scal}, we have
\begin{equation}\label{v-k-bound*}
\left(v_k^{(p-1)/2}+|\nabla v_k|^{(p-1)/(p+1)}\right)(y)\leq 2,\quad |y|\leq k,
\end{equation}
and
\begin{equation}\label{nontri}
  \left(v_k^{(p-1)/2}+|\nabla v_k|^{(p-1)/(p+1)}\right)(0)=\lambda_k M_k(x_k)=1.
\end{equation}
The assumption \eqref{assum-f} of $f$ implies that there exists a constant $C>0$ such that
\begin{equation*}
  -C\leq f(s)\leq C\left(1+s^p\right), \quad s\geq 0.
\end{equation*}
It follows from the boundedness of $v_k$ and $\lambda_k\to 0$ as $k\to\infty$ that
\begin{equation*}
  -C\lambda_k^{2p/(p-1)}\leq\lambda_k^{2p/(p-1)}f\left(\lambda_k^{-2/(p-1)}v_k\right)\leq C,\quad |y|\leq k
\end{equation*}
for $k$ large enough. Noting that $\lambda_k^{-(p+1)/(p-1)}\to\infty$ as $k\to\infty$, from the assumption \eqref{assum-g} on $g$ and the boundedness of $\nabla v_k$, we obtain when $k$ large enough,
\begin{equation*}
 0\leq \lambda_k^{2p/(p-1)}g\left(\lambda_k^{-(p+1)/(p-1)}\nabla v_k\right)\leq 2\mu|\nabla v_k|^{2p/(p+1)}\leq C(p,\mu),\quad |y|\leq k\ \text{with}\ \nabla v_k(y)\neq0,
\end{equation*}
and
\begin{equation*}
  -C\lambda_k^{2p/(p-1)}\leq \lambda_k^{2p/(p-1)}g(0)\leq C, \quad |y|\leq k\ \text{with}\ \nabla v_k(y)=0.
\end{equation*}

By using the elliptic regularity theory, we deduce that there exists $r\in (0,1)$ such that $v_k$ is bounded in $C^{1,r}_{\rm loc}\left(\mathbb R^N\right)$.  Then there exists a $v\geq 0$, up to a subsequence, such that $v_k$ converges $v$ in $C^1_{\rm loc}\left(\mathbb R^N\right)$ which satisfies $-\Delta v\geq 0$ in $\mathbb R^N$. Therefore, by the strong maximum principle, $v$ is nontrivial and $v>0$.

We now proceed with the proof of Theorems \ref{them:uni-esti} and \ref{them:f-g} individually. Under assumptions in Theorem \ref{them:uni-esti}, $v$ is a nontrivial nonnegative bounded solution to \eqref{eq1} in $\mathbb R^N$ with $M\geq M_0$. This contradicts Corollary \ref{corol-u-bound} and hence proves Theorem \ref{them:uni-esti} (i). As for Theorem \ref{them:uni-esti} (ii), it is a direct consequence of Proposition \ref{pro:q=criti}.

For Theorem \ref{them:f-g}, since $v>0$ is nontrivial,  using the assumption \eqref{assum-f} on $f$, we find that for a fixed $y\in\mathbb R^N$, $v_k(y) \geq v(y)/2>0$ for $k$ large enough, and
\begin{equation*}
   \lambda_k^{2p/(p-1)}f\left(\lambda_k^{-2/(p-1)}v_k(y)\right)=v_k^p(y)s_k^{-p}f(s_k)\to lv^p(y)\quad \text{\rm as}\ k\to \infty,
\end{equation*}
where $s_k:=\lambda_k^{-2/(p-1)}v_k(y)\to \infty$ as $k\to\infty$. For a given $y\in \mathbb R^N$ such that $\nabla v(y)\neq 0$, there exists a subsequence $v_k$, satisfying  $\nabla v_k(y)\neq 0$ and $\nabla v_k(y)\to \nabla v(y)$ as $k\to \infty$. By the assumption \eqref{assum-g} on $g$, we have
\begin{equation*}
\lambda_k^{2p/(p-1)}g\left(\lambda_k^{-(p+1)/(p-1)}\nabla v_k(y)\right)=\bar s_k^{-2p/(p+1)}g\left(\bar s_k\nabla v_k(y)\right)\to \mu|\nabla v(y)|^{2p/(p+1)}\quad \text{as}\ k\to\infty,
\end{equation*}
where $\bar s_k:=\lambda_k^{-(p+1)/(p-1)}\to\infty$ as $k\to\infty$. On the other hand, if $\nabla v(y)=0$ and there exists a sequence of $v_k$ such that $\nabla v_k(y)\neq 0$ with $\nabla v_k(y)\to 0$ as $k\to \infty$, the arguments above remains valid. Consequently, it follows from \eqref{v-k-bound*} that $v$ is a positive nontrivial bounded solution to
\begin{equation*}
  -\Delta v=lv^p+\mu|\nabla v|^{\frac{2p}{p+1}}\quad \text{in}\ \mathbb R^N.
\end{equation*}
Setting $w=l^{1/(p-1)}v$, then $w$ is the positive nontrivial bounded solution of
\begin{equation*}
  -\Delta w=w^p+\mu l^{-\frac{1}{p+1}}|\nabla w|^{\frac{2p}{p+1}}\quad \text{in}\ \mathbb R^N.
\end{equation*}
When $1<p<(N+3)/(N-1)$, it contradicts to Theorem \ref{them:p,q<n+3,n+1}. In addition, when $p\geq (N+3)/(N-1)$ with $\mu l^{-1/(p+1)}\geq M_0$, it contradicts to Corollary \ref{corol-u-bound}. This proves Theorem \ref{them:f-g}.
\hfill$\Box$

\vskip 3mm
\noindent{\bf Conflict of interest.} {No potential conflict of interest was reported by the authors.}

\vskip 3mm
\noindent{\bf Data availability.} {No data was used for the research described in the article.}

\vskip 3mm
\noindent{\bf Acknowledgments.} {The research of W.G. Liang is supported by the  Fundamental  Research  Funds  for  the Central Universities (No. xzy022025046). The research of Z.C. Zhang is partially supported by the National Natural
Science Foundation of China (Nos. 12271423 and 12671248).

\end{document}